\documentclass[12pt]{amsart}
\usepackage{a4wide}
\usepackage{graphicx}
\usepackage{color}
\usepackage{amsmath}
\usepackage{bm} 
\usepackage{enumitem} 
\usepackage{mathabx} 
\allowdisplaybreaks

\let\pa\partial
\let\na\nabla
\let\eps\varepsilon
\newcommand{\N}{{\mathbb N}}
\newcommand{\R}{{\mathbb R}}
\newcommand{\diver}{\operatorname{div}}

\newcommand{\E}{\mathcal E}
\newcommand{\Q}{{\mathbb Q}}
\newcommand{\DD}{{\mathbb D}}
\newcommand{\MM}{{\bm{M}^0}}
\newcommand{\wMM}{{\widehat{\bm{M}}}^0}

\newtheorem{theorem}{Theorem}
\newtheorem{lemma}[theorem]{Lemma}
\newtheorem{proposition}[theorem]{Proposition}
\newtheorem{remark}[theorem]{Remark}

\newtheorem{definition}{Definition}

\usepackage[colorlinks=true,citecolor=magenta,linkcolor=magenta]{hyperref}

\begin{document}

\title[Reaction-cross-diffusion systems of Maxwell--Stefan type]{Exponential time decay
of solutions \\ to reaction-cross-diffusion systems \\ of Maxwell--Stefan type}

\author[E. S. Daus]{Esther S. Daus}
\address{Institute for Analysis and Scientific Computing, Vienna University of
	Technology, Wiedner Hauptstra\ss e 8--10, 1040 Wien, Austria}
\email{esther.daus@tuwien.ac.at}

\author[A. J\"ungel]{Ansgar J\"ungel}
\address{Institute for Analysis and Scientific Computing, Vienna University of
	Technology, Wiedner Hauptstra\ss e 8--10, 1040 Wien, Austria}
\email{juengel@tuwien.ac.at}

\author[B. Q. Tang]{Bao Quoc Tang}
\address{Institute of Mathematics and Scientific Computing, University of Graz, 
  Heinrichstrasse 36, 8010 Graz, Austria}
\email{quoc.tang@uni-graz.at}

\date{\today}

\thanks{The first and second authors acknowledge partial support from
the Austrian Science Fund (FWF), grants P27352, P30000, F65, and W1245.
The last author was partially supported by the International Research Training Group 
IGDK 1754 and NAWI Graz.
This work was carried out during the visit of the first author to the University 
of Graz and of the the third author to the Vienna University of Technology. 
The hospitality of the universities is greatly acknowledged.}

\begin{abstract}
The large-time asymptotics of weak solutions to 
Maxwell--Stefan diffusion systems for chemically reacting fluids with different
molar masses and reversible reactions are investigated. 
The diffusion matrix of the system is generally neither 
symmetric nor positive definite, but the equations admit a formal gradient-flow 
structure which provides entropy (free energy) estimates. The main result is 
the exponential decay to the unique equilibrium with a rate that is constructive up to
a finite-dimensional inequality. The key elements of the proof are the existence of 
a unique detailed-balanced equilibrium and
the derivation of an inequality relating the entropy and the entropy production. 
The main difficulty comes from the fact that the reactions are represented by 
molar fractions while the conservation laws hold for the concentrations.
The idea is to enlarge the space of $n$ partial concentrations 
by adding the total concentration,
viewed as an independent variable, thus working with $n+1$ variables. 
Further results concern
the existence of global bounded weak solutions to the parabolic
system and an extension of the results to complex-balanced systems. 
\end{abstract}



\maketitle

\tableofcontents


\section{Introduction}

The analysis of the large-time behavior of dynamical networks is important to
understand their stability properties. Of particular interest are
reversible chemical reactions interacting with diffusion. While there
is a vast literature on the large-time asymptotics of reaction-diffusion systems,
much less results are available for reaction systems with cross-diffusion terms.
Such systems arise naturally in multicomponent fluid modeling and population dynamics
\cite{Jue16}. In this paper, we prove the exponential decay of solutions to
reaction-cross-diffusion systems of Maxwell--Stefan form by combining recent
techniques for cross-diffusion systems \cite{Jue15} and reaction-diffusion
equations \cite{FeTa17a}. The main feature of our result is that the decay rate 
is constructive up to a finite-dimensional inequality and that the result
holds for detailed-balanced or complex-balanced systems.

\subsection{Model equations}

We consider a fluid consisting of $n$ constituents $A_i$ with mass
densities $\rho_i(z,t)$ and molar masses $M_i$, which are diffusing according
to the diffusive fluxes $\bm{j}_i(z,t)$ and reacting in the following
reversible reactions,
$$
  \alpha_1^a A_1 + \cdots + \alpha_n^a A_n \leftrightharpoons
	\beta_1^a A_1 + \cdots + \beta_n^a A_n \quad\mbox{for }a=1,\ldots,N,
$$
where $\alpha_i^a$ and $\beta_i^a$ are the stoichiometric coefficients.
The evolution of the fluid is assumed to be governed by partial mass balances
with Maxwell--Stefan relations for the diffusive fluxes,
\begin{equation}\label{1.eq}
  \pa_t\rho_i + \diver\bm{j}_i = r_i(\bm{x}), \quad 
	 \na x_i = -\sum_{j=1}^n\frac{\rho_j\bm{j}_i-\rho_i\bm{j}_j}{c^2M_iM_jD_{ij}},
	\quad i=1,\ldots,n, 
\end{equation}
where $x_i=c_i/c$ are the molar fractions, $c_i=\rho_i/M_i$ the partial concentrations,
$M_i$ the molar masses,
$c=\sum_{i=1}^n c_i$ the total concentration, and $D_{ij}=D_{ji}>0$ 
are the diffusivities. The physical quantities are summarized in Table \ref{table}.
The reactions are described by the
mass production terms $r_i$ depending on $\bm{x}=(x_1,\ldots,x_n)$
using mass-action kinetics,
\begin{equation}\label{1.reac}
  r_i(\bm{x}) = M_i\sum_{a=1}^N(\beta_i^a-\alpha_i^a)
	(k_f^a\bm{x}^{\bm{\alpha^a}} - k_b^a\bm{x}^{\bm{\beta^a}})
	\quad\mbox{with }\bm{x}^{\bm{\alpha^a}}:=\prod_{i=1}^n x_i^{\alpha_i^a},
\end{equation}
where $k_f^a>0$ and $k_b^a>0$ are the forward and backward reaction rate constants,
respectively, and $\bm{\alpha}^a=(\alpha_1^a,\ldots,\alpha_n^a)$ and 
$\bm{\beta}^a=(\beta_1^a,\ldots,\beta_n^a)$ with $\alpha_i^a$, $\beta_i^a\in\{0\}
\cup[1,\infty)$ are the vectors of the stoichiometric coefficients.

\begin{table}[ht]
\begin{tabular}{ll}
\hline
$\rho_i$ &: partial mass density of the $i$th species \\[1mm]
$\rho=\sum_{i=1}^n \rho_i$ &: total mass density \\[1mm]
$\bm{j}_i$ &: partial particle flux of the $i$th species \\[1mm]
$M_i$ &: molar mass of the $i$th species \\[1mm]
$c_i = \rho_i/M_i$ &: partial concentration of the $i$th species \\[1mm]
$c=\sum_{i=1}^n c_i$ &: total concentration \\[1mm]
$x_i = c_i/c$ &: molar fraction \\[1mm]
\hline
\end{tabular}
\vskip2mm
\caption{Overview of the physical quantities.}
\label{table}
\end{table}

Equations \eqref{1.eq} are solved in the bounded domain $\Omega\subset\R^d$
($d\ge 1$) subject to the no-flux boundary and initial conditions
\begin{equation}\label{1.bic}
  \bm{j}_i\cdot\nu=0\mbox{ on }\pa\Omega,\quad \rho_i(\cdot,0)=\rho_i^0
	\quad\mbox{in }\Omega,\ i=1,\ldots,n.
\end{equation}
To simplify, we assume that $\Omega$ has unit measure, i.e. $|\Omega| = 1$.

System \eqref{1.eq}-\eqref{1.reac} models a multicomponent fluid in an isothermal regime
with vanishing barycentric velocity. Equations \eqref{1.eq} for $\na x_i$
can be derived from the Boltzmann equations for mixtures in the diffusive
limit and with well-prepared initial conditions \cite{BGS15,HuSa17,HuSa17a} or from the
reduced force balances with the partial momentum productions being proportional
to the partial velocity differences \cite[Section~14]{BoDr15}.

We assume that the total mass is conserved and that the mixture is at rest,
i.e., $\sum_{i=1}^n\rho_i=1$ and $\sum_{i=1}^n\bm{j}_i=0$. This implies that
\begin{equation}\label{1.ctm}
  \sum_{i=1}^n r_i(\bm{x})=0 \quad\mbox{for all }\bm{x}=(x_1,\ldots,x_n)\in\R_+^n,
\end{equation}
where $\R_+=(0,\infty)$. Furthermore, we assume that the system of reactions satisfies
a detailed-balanced condition, meaning that there exists a positive homogeneous
equilibrium $\bm{x}_\infty\in\R_+^n$ such that 
\begin{equation}\label{1.db}
  k_f^a\bm{x}_\infty^{\bm{\alpha}^a} = k_b^a\bm{x}_\infty^{\bm{\beta}^a}
	\quad\mbox{for all }a=1,\ldots,N. 
\end{equation}
Roughly speaking, a system is under detailed balance if any forward reaction 
is balanced by the corresponding backward reaction at equilibrium. Condition \eqref{1.db}
does not give a unique but instead a manifold of detailed-balanced equilibria,
\begin{equation}\label{set_equi}
	\E = \big\{\bm{x}_\infty\in \R_+^n:\ k_f^a\bm{x}_\infty^{\bm{\alpha}^a} 
	= k_b^a\bm{x}_\infty^{\bm{\beta}^a}\quad\mbox{for all } a=1,\ldots,N\big\}.
\end{equation}
To uniquely identify the detailed-balanced equilibrium, we need to take into account
the conservation laws (meaning that certain linear combinations of the concentrations
are constant in time). This is discussed in detail below. 
We are also able to consider complex-balanced systems; see Section \ref{sec.complex}.

The aim of this paper is to prove that under these conditions, there exists
a unique positive detailed-balanced (or complex-balanced) equilibrium 
$\bm{x}_\infty=(x_{1\infty},\ldots,x_{n\infty})\in\R_+^n$ such that 
$$
  \sum_{i=1}^n\|x_i(t)-x_{i\infty}\|_{L^p(\Omega)}
	\le C(\bm{x}^0,\bm{x}_\infty)e^{-\lambda t/(2p)}, \quad t>0,\ p\ge 1,
$$
where $\bm{x}^0=\bm{x}(0)$ and 
the constant $\lambda>0$ is constructive up to a finite-dimensional
inequality. Before we make this result precise, we review the
state of the art and explain the main difficulties and key ideas.

\subsection{State of the art}

The research of the large-time asymptotics of general reaction-diffusion systems
with diagonal diffusion, modeling chemical reactions,
has experienced a dramatic scientific progress in recent years. 
One reason for this progress is due to new developments of so-called entropy 
methods. Classical methods include linearized stability techniques, spectral theory,
invariant region arguments, and Lyapunov stability; see, e.g., \cite{CHS78,FHM97}.
The entropy method is a genuinely nonlinear approach without using any kind of
linearization, it is rather robust against model variations,
and it is able to provide explicitly computable decay rates.
The first related works date back to the 1980s \cite{Gro83,Gro86}. The obtained 
results are restricted to two space dimensions and do not provide explicit estimates, 
since the proofs are based on contradiction arguments.
First applications of the entropy method that provide explicit rates and 
constants were concerned with particular cases, like two-component systems 
\cite{DeFe06}, four-component systems \cite{DeFe14}, 
or multicomponent linear systems \cite{DFFM08}.
Later, nonlinear reaction networks with an arbitrary number of 
chemical substances were considered \cite{FeTa17,MHM15}. 
Exponential convergence of close-to-equilibrium solutions to quadratic
reaction-diffusion systems with detailed balance was shown in \cite{CaCa17}.
Reaction-diffusion systems without
detailed balance \cite{FPT17} and with complex balance \cite{DFT17,Mie17,Tan17} 
were also thoroughly investigated. The convergence to equilibrium was proven
for rather general solution concepts, like
very weak solutions \cite{PSU17} and renormalized solutions \cite{FeTa17a}. 

The large-time behavior of solutions to cross-diffusion systems is less
studied in the literature. The convergence to equilibrium was shown for the
Shigesada--Kawasaki--Teramoto population model with Lotka--Volterra terms in 
\cite{Shi06,WeFu09} without any rate and in \cite{ChJu06} without reaction terms.
The exponential decay of solutions to volume-filling population systems, again
without reaction terms, was proved in \cite{ZaJu17}. 

A number of articles is concerned with the large-time asymptotics in 
Maxwell--Stefan systems. For global existence results on these systems, we refer
to \cite{HMPW17,JuSt12,MaTe15}. In \cite{JuSt12},
the exponential decay to the homogeneous state state is shown with vanishing
reaction rates and same molar masses. The result was generalized to different molar 
masses in \cite{ChJu15}, but still without reaction terms.  
The convergence to equilibrium was proved in 
\cite[Theorem 9.7.4]{Gio99} and \cite[Theorem 4.3]{HMPW17} under the condition
that the initial datum is close to the equilibrium state.
The work \cite{HMPW17} also addresses the exponential convergence to a homogeneous
equilibrium assuming (i) global existence of strong solutions and (ii) uniform-in-time
strict positivity of the solutions (see Prop.~4.4 therein). 
A similar result, but for two-phase systems, was proved in \cite{BoPr17}. 
The novelty in our paper is that we provide also a global existence proof 
(which avoids assumption (i)) and that we replace the strong assumption (ii) 
by a natural condition on the reactions, namely that there exist no equilibria
on $\pa\R_+^n$. We note that there exists a large class of chemical reaction networks,
called {\em concordant networks}, which possess no boundary equilibria
\cite[Theorem 2.8(ii)]{SF13}.

\subsection{Key ideas}

The analysis of the Maxwell--Stefan equations \eqref{1.eq} is rather delicate.
The first difficulty is that 
the fluxes are not given as linear combinations of the gradients
of the mass fractions, which makes it necessary to invert the flux-gradient relations
in \eqref{1.eq}. However, summing the equations for $\na x_i$ in \eqref{1.eq}
for $i=1,\ldots,n$, we see that the Maxwell--Stefan equations are linear dependent,
and we need to invert them on a subspace \cite{Bot11}. The idea is
to work with the $n-1$ variables $\bm{\rho}'=(\rho_1,\ldots,\rho_{n-1})^\top$
by setting $\rho_n=1-\sum_{i=1}^{n-1}\rho_i$, i.e., the mass density 
of the last component (often the solvent) is computed from the other mass densities. 
Then there exists a diffusion
matrix $\mathbb{A}(\bm{\rho}')\in\R^{(n-1)\times(n-1)}$ such that system \eqref{1.eq}
can be written as
\begin{equation}\label{1.eq2}
  \pa_t\bm{\rho}' - \diver(\mathbb{A}(\bm{\rho}')\na\bm{x}') = \bm{r}'(\bm{x}),
\end{equation}
where $\bm{x}'=(x_1,\ldots,x_{n-1})^\top$ and $\bm{r}'=(r_1,\ldots,r_{n-1})^\top$.
The matrix $\mathbb{A}(\bm{\rho}')$ is generally neither symmetric nor positive definite.
However, equations \eqref{1.eq2} exhibit a formal gradient-flow structure 
\cite{JuSt12}. This means the following: We introduce the so-called (relative)
entropy density
\begin{equation}\label{1.h}
  h(\bm{\rho}') = c\sum_{i=1}^n x_i\ln\frac{x_i}{x_{i\infty}}, \quad
	\mbox{where }\rho_n=1-\sum_{i=1}^{n-1} \rho_i,
\end{equation}
and the entropy variable $\bm{w}=(w_1,\ldots,w_{n-1})^\top$ with $w_i=\pa h/\pa\rho_i$.
Here, $\bm{x}_\infty \in \E$ is an arbitrary detailed-balanced equilibrium. 
We associate to the entropy density the relative entropy
(or free energy)
\begin{equation}\label{1.ent}
  E[\bm{x}|\bm{x}_\infty] = \int_\Omega h(\bm{\rho}')dz
	= \sum_{i=1}^n\int_\Omega cx_i\ln\frac{x_i}{x_{i\infty}}dz.
\end{equation}
Denoting by $h''(\bm{\rho}')$ the Hessian of $h$ with respect to $\bm{\rho}'$,
equation \eqref{1.eq2} is equivalent to
\begin{equation}\label{1.eqw}
  \pa_t\bm{\rho}' - \diver(\mathbb{B}(\bm{w})\na\bm{w}) = \bm{r}'(\bm{x}),
\end{equation}
where $\mathbb{B}(\bm{w})=\mathbb{A}(\bm{\rho}')h''(\bm{\rho}')^{-1}$ is symmetric
and positive definite \cite[Lemma 10 (iv)]{ChJu15} and $\bm{\rho}'$ and $\bm{x}$ are
functions of $\bm{w}$. The elliptic operator can be formulated as 
$\mathbb{K}\operatorname{grad}h(\bm{\rho}')$, where 
$\mathbb{K}\xi=\diver(\mathbb{B}\na\xi)$ is the Onsager operator and grad
is the functional derivative. This formulation motivates the notion ``gradient-flow
structure''.

The second difficulty comes from the fact that the
cross-diffusion coupling prevents the use of standard tools like maximum principles 
and regularity theory. In particular, it is not clear how to prove
lower and upper bounds for the mass densities or molar fractions. 
Surprisingly, this problem can be also solved by the transformation to entropy
variables. Indeed, the mapping $(0,1)^{n-1}\to\R^{n-1}$, $\bm{\rho}'\mapsto\bm{w}$,
can be inverted, and the image $\bm{\rho}'(\bm{w})$ lies in $(0,1)^{n-1}$ and
satisfies $1-\sum_{i=1}^{n-1}\rho_i<1$. If all molar masses are equal, $M=M_i$,
the inverse function can be written explicitly as
$\rho_i(\bm{w})=\exp(Mw_i)(1+\sum_{j=1}^{n-1}\exp(Mw_j))^{-1}$; for the general
case see Lemma \ref{lem.inv} below. This yields the positivity and $L^\infty$
bounds for $\rho_i$ without the use of a maximum principle. To make this argument
rigorous, we first need to solve \eqref{1.eqw} for $\bm{w}$ and then to conclude that
$\bm{\rho}'=\bm{\rho}'(\bm{w})$ solves \eqref{1.eq}.

Summarizing, the entropy helps us to ``symmetrize'' system \eqref{1.eq} 
and to derive $L^\infty$ bounds. There is a further benefit:
The entropy is a Lyapunov functional along solutions to the detailed-balanced
system \eqref{1.eq}.
Indeed, a formal computation shows the following relation (a weaker discrete version 
is made rigorous in the proof of Theorem \ref{thm.ex}),
\begin{equation}\label{1.dEdt}
  \frac{d}{dt}E[\bm{x}|\bm{x}_\infty] + D[\bm{x}] = 0,\quad t>0,
\end{equation}
where the entropy production
\begin{equation}\label{1.ep}
  D[\bm{x}] = \sum_{i,j=1}^{n-1}\int_\Omega B_{ij}(\bm{w})\na w_i\cdot\na w_j dz
	+ \sum_{a=1}^N\int_\Omega(k_f^a\bm{x}^{\bm{\alpha}^a}-k_b^a\bm{x}^{\bm{\beta}^a})
	\ln\frac{k_f^a \bm{x}^{\bm{\alpha}^a}}{k_b^a \bm{x}^{\bm{\beta}^a}}dz
\end{equation}
is nonnegative (due to Lemmas \ref{lem.B} and \ref{lem.r}). Here,
$B_{ij}$ are the coefficients of the matrix $\mathbb{B}$.
Exponential decay follows if the entropy entropy-production inequality
\begin{equation}\label{1.eep}
  D[\bm{x}] \ge\lambda E[\bm{x}|\bm{x}_\infty]
\end{equation}
holds for all suitable functions $\bm{x}$ and for some $\lambda>0$. 
Note that this functional inequality does not hold for all 
detailed-balanced equilibria, but only for those who satisfy certain conservation 
laws. The existence and uniqueness of such equilibria is proved in Theorem \ref{thm.equi}.
Inserting inequality \eqref{1.eep} into \eqref{1.dEdt} yields 
$$
  \frac{d}{dt}E[\bm{x}|\bm{x}_\infty] + \lambda E[\bm{x}|\bm{x}_\infty] \le 0, 
	\quad t>0,
$$
and Gronwall's inequality allows us to conclude that
$$
  E[\bm{x}(t)|\bm{x}_\infty] \le E[\bm{x}(0)|\bm{x}_\infty] e^{-\lambda t}, 
	\quad t>0.
$$
By a variant of the Csisz\'ar--Kullback--Pinsker inequality (Lemma \ref{lem.CKP}),
this gives exponential decay in the $L^1$ norm with rate $\lambda/2$ and,
by interpolation, in the $L^p$ norm with rate $\lambda/(2p)$ for all $1\le p<\infty$.
An important feature of this result is that the constant $\lambda$ is constructive
up to a finite-dimensional inequality. 

The cornerstone of the convergence to equilibrium is to prove inequality \eqref{1.eep}. 
In comparison to previous results for reaction-diffusion systems, e.g.\ 
\cite{FeTa17,MHM15}, the difference here is that the reactions are defined in terms 
of molar fractions, while the conservation laws are written in terms of concentrations. 
This difference causes the main difficulty in proving \eqref{1.eep}, except in 
very special cases, e.g., when all molar masses are equal (in this case, the 
molar fraction and concentration are proportional) or in case of equal homogeneities 
(see Section \ref{sec.eqhom}). Naturally, one could express the molar fractions by the 
concentrations, i.e. $x_i = c_i/(\sum_{i=1}^{n}c_i)$, but this extremely complicates
the formulation of the entropy production $D[\bm{x}]$, 
which in turn makes the analysis of \eqref{1.eep} 
inaccessible. The key idea here is to introduce the total concentration 
$c = \sum_{i=1}^{n}c_i$ as an independent variable and to rewrite $D[\bm{x}]$ in 
terms of $x_i = c_i/c$. This, in combination with an estimate for 
$E[\bm{x}|\bm{x}_\infty]$ in terms of $c_i$ and $c$, allows us to adapt the ideas 
from previous works on reaction-diffusion systems to finally obtain the 
desired inequality \eqref{1.eep}.


\subsection{Main results}\label{sec.main}

Our main result is the exponential convergence to equilibrium. For this, we need
to show some intermediate results. The existence of solutions to 
\eqref{1.eq}, \eqref{1.bic} was shown in \cite{ChJu15} without
reaction terms. Therefore, we prove the global existence of bounded weak solutions
to \eqref{1.eq}, \eqref{1.bic} with reaction terms \eqref{1.reac}. The proof
follows that one in \cite{ChJu15} but the estimates related to the reaction terms
are different. A key step is the proof of the monotonicity of $\bm{w}\mapsto
\sum_{i=1}^{n-1} r_i(\bm{x})$; see Lemma \ref{lem.r}. 

Second, we derive the conservation laws satisfied
by the solutions to \eqref{1.eq} (Lemma \ref{lem.cl}) 
and prove the existence of a positive detailed-balanced equilibrium $\bm{x}_\infty$
satisfying \eqref{1.db} and the conservation laws (Theorem \ref{thm.equi}). 
The existence of unique equilibrium states for chemical reaction networks
is well studied in the literature (see, e.g., \cite{Fei95}), but not in the 
present framework. One difficulty is the additional constraint $\sum_{i=1}^n x_i=1$,
which significantly complicates the analysis.
The key idea for the existence of a unique detailed-balanced equilibrium
is to analyze systems in the partial concentrations $c_1,\ldots,c_n$ {\em and} 
the total concentration $c$, considered as an independent variable. The increase
of the dimension of the system from $n$ to $n+1$ allows us to apply geometric
arguments and a result of Feinberg \cite{Fei95} to achieve the claim.

Third, we prove the entropy entropy-production inequality \eqref{1.eep}
(Prop.\ \ref{prop.eqhom} and \ref{prop.uneqhom}).
The proof follows basically from \cite[Lemma 2.7]{FeTa17a} when
the stoichiometric coefficients satisfy $\sum_{i=1}^n\alpha_i^a=\sum_{i=1}^n\beta_i^a$
for all $a=1,\ldots,N$, since this property allows us to replace the
molar fractions $x_i$ by the concentrations $c_i$. If the property is
not fulfilled, we work again in the augmented space of concentrations
$(c_1,\ldots,c_n,c)$. One step of the proof (Lemma \ref{lem.Dreac.ge})
requires the proof of an inequality whose constant is constructive only up
to a finite-dimensional inequality. We believe that for concrete systems,
this constant can be computed in a constructive way. We present such an 
example in Section \ref{sec.exam}.

Before stating the main theorem, we need some notation. Let 
$$
  \mathbb{W} = (\bm{\beta}^a-\bm{\alpha}^a)_{a=1,\ldots,N}\in\R^{n\times N},
$$
be the Wegscheider matrix (or stoichiometric coefficients matrix)
and set $m=\dim\operatorname{ker}(\mathbb{W}^\top)$ $>0$.
We choose a matrix $\Q\in\R^{m\times n}$ whose rows form a basis of 
$\operatorname{ker}(\mathbb{W}^\top)$.
Let $\MM\in\R_+^m$ be the initial mass vector, 
which depends on $\bm{c}^0$ (see Lemma \ref{lem.cl}) and
let $\bm{\zeta}\in\R^{1\times m}$ be a row
vector satisyfing $\bm{\zeta}\Q=(M_1,\ldots,M_n)$ and $\bm{\zeta}\MM=1$.
We show in Lemma \ref{lem.zeta} that such a vector $\bm{\zeta}$ always exists.
Its appearance comes from the constraint $\sum_{i=1}^n x_i=1$; such a vector is
not needed in reaction-diffusion systems like in \cite{FeTa17a}.
Given $\MM\in\R_+^m$ such that $\bm{\zeta}\MM=1$, we prove in Section \ref{sec.dbc}
that there exists a unique positive {\em detailed-balanced equilibrium} 
$\bm{x}_\infty\in\E$ satisfying
\begin{equation}\label{1.dbc}
	\Q\bm{c}_\infty = \MM, \quad
	\sum_{i=1}^n x_{i\infty}=1,
\end{equation}
where the components of $\bm{c}_\infty$ are given by
$c_{i\infty}=x_{i\infty}/\sum_{i=1}^n M_ix_{i\infty}$.
The first expression in \eqref{1.dbc} are the conservation laws, while
the second one is the normalization condition.

Note that besides the unique positive detailed-balanced equilibrium (for a 
fixed initial mass vector), there could exist possibly infinitely many
{\em boundary equilibria}, i.e.\ $\bm{x}^*\in\pa\E$ such that $\bm{x}^*$
solves \eqref{1.dbc}. We need to exclude such equilibria. For a discussion of
boundary equilibria and the Global Attractor Conjecture, we refer to Remark
\ref{rem.be}.

\medskip
\begin{enumerate}[label=(A\theenumi),ref=A\theenumi] 
\item\label{A1} Data: $\Omega\subset\R^d$ with $d\ge 1$ is a bounded domain 
with Lipschitz boundary,
$T>0$, and $D_{ij}=D_{ji}>0$ for $i,j=1,\ldots,n$, $i\neq j$.

\item\label{A2} Detailed-balance condition: $\E \neq \emptyset$, 
where $\E$ is defined in \eqref{set_equi}.

\item\label{A3} Initial condition: $\bm{\rho}^0\in L^1(\Omega;\R^n)$
with $\rho_i^0\ge 0$, $\sum_{i=1}^n\rho_i^0=1$, and the initial entropy is finite,
$\int_\Omega h({\bm{\rho}^0}')dz<\infty$, 
where $h$ is defined in \eqref{1.h} with some $\bm{x}_\infty\in \E$.
\end{enumerate}

\medskip

The main result is as follows.

\begin{theorem}[Convergence to equilibrium]\label{thm.main}
Let Assumptions \eqref{A1}-\eqref{A3} hold. Let $\MM \in \R_+^m$ be a positive 
initial mass vector satisfying $\bm{\zeta} \MM = 1$. Then

{\rm (i)} There exists a global bounded weak solution 
$\bm{\rho}=(\rho_1,\ldots,\rho_n)^\top$ 
to \eqref{1.eq}-\eqref{1.reac} in the sense of Theorem \ref{thm.ex} below.

{\rm (ii)} There exists a unique $\bm{x}_\infty\in\E$ satisfying \eqref{1.dbc},
where the set of equilibria $\E$ is defined in \eqref{set_equi}.

{\rm (iii)} Assume in addition that the system \eqref{1.eq}-\eqref{1.reac} has no 
boundary equilibria. Then there exist constants $C>0$ and $\lambda>0$, which 
are constructive up to a finite-dimensional inequality, such that, if $\bm{\rho}^0$ 
satisfies additionally $\mathbb Q\int_{\Omega}\bm{c}^0dz = \MM$, the following 
exponential convergence to equilibrium holds:
$$
  \sum_{i=1}^n\|x_i(t)-x_{i\infty}\|_{L^p(\Omega)} \le Ce^{-\lambda t/(2p)}
	\big(E[\bm{x}^0|\bm{x}_\infty]\big)^{1/(2p)} ,\quad t>0,
$$
where $1\le p <\infty$,  $x_i=\rho_i/(cM_i)$ with $c=\sum_{i=1}^n\rho_i/M_i$, 
$E[\bm{x}|\bm{x}_\infty]$ is the relative
entropy defined in \eqref{1.ent}, $\bm{\rho}$
is the solution constructed in {\rm (i)}, and $\bm{x}_\infty$ is 
constructed in {\rm (ii)}.
\end{theorem}

\begin{remark}[Complex balance]\label{rem.cb}\rm
We show in Theorem \ref{thm.equi} that 
system \eqref{1.eq} with the reaction terms \eqref{1.reac} possesses a unique
positive detailed-balance equilibrium. This means that we have assumed the 
reversibility of the reaction system. 
This assumption is rather strong, and it is well known
in chemical reaction network theory that it can be significantly
generalized to complex-balanced systems. Here, the balance is not assumed to hold 
for any elementary reaction step but only for the total in-flow and total 
out-flow of each chemical complex.
We are able to extend our results to this situation as well, considering
the reaction terms \eqref{5.reac}; 
see Theorem \ref{thm.main2} in Section \ref{sec.complex}.

Clearly, any detailed-balanced equilibrium is also a complex-balanced equilibrium,
and Theorem \ref{thm.main} is included in Theorem \ref{thm.main2}.
However, to make the proofs as accessible as possible, we prefer to present
the detailed-balanced case in full detail and sketch the extension to 
complex-balanced systems. 
\qed
\end{remark}

The paper is organized as follows. Part (i) of Theorem \ref{thm.main} 
is proved in Section \ref{sec.ex}. 
In Section \ref{sec.time1}, the conservation laws
are derived, the existence of a detailed-balanced equilibrium and the
entropy entropy-production inequality \eqref{1.eep} are proved, and 
the convergence result is shown. Section \ref{sec.exam} is concerned with
a specific example for which the constant in the entropy 
entropy-production inequality can be computed explicitly. The results
are extended to complex-balanced systems in Section \ref{sec.complex}.
Finally, we prove the technical Lemma \ref{lem.tech} in the appendix.


\subsection{Notation}

We use the following notation:

\begin{itemize}[leftmargin=10mm]
\item Bold letters indicate vectors in $\R^n$ (e.g.\ $\bm{c}=(c_1,\ldots,c_n)^\top$).
\item Normal letters denote the sum of all the components of the corresponding 
letter in bold font (e.g.\ $c=\sum_{i=1}^n c_i$).
\item Primed bold letters signify that the last component is removed from the 
original vector (e.g.\ $\bm{c}'=(c_1,\ldots,c_{n-1})^\top$). 
\item Overlined letters usually denote integration over $\Omega$ (e.g.\
$\overline{\bm{c}}=\int_\Omega\bm{c}dz$ or $\overline{c_i}=\int_\Omega c_idz$). 
\item If $f:\R\to\R$ is a function and $\bm{c}\in\R^n$ a vector, the expression
$f(\bm{c})$ denotes the vector $(f(c_1),\ldots,f(c_n))^\top$. 
\item Let $\bm{x}$, $\bm{\alpha}\in(0,\infty)^n$. The expression
$\bm{x}^{\bm{\alpha}}$ equals the product $\prod_{i=1}^n x_i^{\alpha_i}$.
\item Matrices are generally denoted by double-barred capital letters (e.g.\ 
$\mathbb{A}\in\R^{m\times n}$).
\end{itemize}

The inner product in $\R^n$ is denoted by $\langle\cdot,\cdot\rangle$,
$|\Omega|$ is the measure of $\Omega$, and we set $\R_+=(0,\infty)$.
In the estimates, $C>0$ denotes a generic constant with values changing from line 
to line.


\section{Global existence of weak solutions}\label{sec.ex}

We prove part (i) of Theorem \ref{thm.main}. Throughout this section, we fix an 
arbitrary detailed-balanced equilibrium $\bm{x}_\infty \in \E$. Due to \eqref{A2}, 
such a vector $\bm{x}_\infty$ always exists. The existence result is stated
more precisely in the following theorem.

\begin{theorem}[Global existence]\label{thm.ex}
Let Assumptions \eqref{A1}-\eqref{A3} hold. Then there exists a bounded weak solution
$\bm{\rho}=(\rho_1,\ldots,\rho_n)^\top$ to \eqref{1.eq}-\eqref{1.bic} satisfying
$\rho_i\ge 0$, $\sum_{i=1}^n\rho_i=1$ in $\Omega\times(0,T)$ and
$$
  \rho_i\in L^2(0,T;H^1(\Omega)),\ \pa_t\rho_i\in L^2(0,T;H^1(\Omega)'), \quad
	i=1,\ldots,n,
$$
i.e., for all $q_1,\ldots,q_{n-1}\in L^2(0,T;H^1(\Omega))$, 
\begin{equation}\label{2.weak}
  \sum_{i=1}^{n-1}\int_0^T\langle\pa_t\rho_i,q_i\rangle dt
	+ \sum_{i,j=1}^{n-1}\int_0^T\int_\Omega
	A_{ij}(\bm{\rho}')\na x_i\cdot\na q_j dzdt
	= \sum_{i=1}^{n-1}\int_0^T\int_\Omega r_i(\bm{x})q_i dzdt,
\end{equation}
where $\bm{x}=(x_1,\ldots,x_n)^\top$, $x_i=\rho_i/(cM_i)$ for $i=1,\ldots,n-1$,
$x_n=1-\sum_{i=1}^{n-1}x_i$, $c=\sum_{i=1}^n\rho_i/M_i$, and $\mathbb{A}=(A_{ij})$ 
is the diffusion matrix in \eqref{1.eq2}.
\end{theorem}

The proof is similar to that one given in \cite{ChJu15}. Since in that paper,
no reaction terms have been considered, we need to show how these terms can be 
controlled. First, we collect some results. 

\subsection{Preliminary results}

A straightforward computation 
(see \cite[Lemma 5]{ChJu15}) shows that the entropy variables are given by
\begin{equation}\label{2.w}
  w_i = \frac{\pa h}{\pa\rho_i} 
	= \frac{1}{M_i}\ln\frac{x_i}{x_{i\infty}} - \frac{1}{M_n}\ln\frac{x_n}{x_{n\infty}},
	\quad i=1,\ldots,n-1,
\end{equation}
recalling $h$ defined in \eqref{1.h}. Given $\bm{\rho}'=(\rho_1,\ldots,\rho_{n-1})^\top$, 
this formula and the relation $x_i=\rho_i/(cM_i)$
allow us to compute $\bm{w}=(w_1,\ldots,w_{n-1})^\top$. The following lemma states that
the mapping $\bm{\rho}'\mapsto\bm{w}$ can be inverted.

\begin{lemma}\label{lem.inv}
Let $\bm{w}=(w_1,\ldots,w_{n-1})^\top\in\R^{n-1}$ be given. Then there exists a unique
vector $\bm{\rho}'=(\rho_1,\ldots,$ $\rho_{n-1})^\top\in(0,1)^{n-1}$ satisfying 
$\sum_{i=1}^{n-1}\rho_i<1$
such that \eqref{2.w} holds with $\rho_n=1-\sum_{i=1}^{n-1}\rho_i>0$,
$x_i=\rho_i/(cM_i)$ and $c=\sum_{i=1}^n\rho_i/M_i$. Moreover, the function
$\bm{\rho}':\R^{n-1}\to(0,1)^{n-1}$, 
$(w_1,\ldots,w_{n-1})^\top\mapsto \bm{\rho}'(w)=(\rho_1,\ldots,\rho_{n-1})^\top$
is bounded.
\end{lemma}

\begin{proof}
First, we show that there exists a unique vector 
$(x_1,\ldots,x_{n-1})^\top\in(0,1)^{n-1}$ satisfying \eqref{2.w} with 
$x_n=1-\sum_{i=1}^{n-1}x_i>0$ (see \cite[Lemma 6]{ChJu15}). 
Let $z_i:=x_{i\infty}/x_{n\infty}^{M_i/M_n}$.
The function 
$$
  f(s)=\sum_{i=1}^{n-1}z_i(1-s)^{M_i/M_n}\exp(M_iw_i)
$$
is strictly
decreasing in $[0,1]$ and $0=f(1)<f(s)<f(0)=\sum_{i=1}^{n-1}\exp(M_iw_i)z_i$. 
Thus, there exists a unique fixed point $s_0\in(0,1)$ such that $f(s_0)=s_0$. 
Defining $x_i=z_i(1-s_0)^{M_i/M_n}\exp(M_iw_i)$ for $i=1,\ldots,n-1$, we infer that
$x_i>0$, $\sum_{i=1}^{n-1}x_i=f(s_0)=s_0<1$, and \eqref{2.w} holds with 
$x_n:=1-s_0$.

Next, let $(x_1,\ldots,x_{n-1})^\top\in(0,1)^{n-1}$ 
and $x_n:=1-\sum_{i=1}^{n-1}x_i>0$ be given 
and define $\rho_i=cM_ix_i$, where $c=1/(\sum_{i=1}^{n}M_ix_i)$. Then
$(\rho_1,\ldots,\rho_{n-1})^\top\in(0,1)^{n-1}$ is the unique vector satisfying
$\rho_n=1-\sum_{i=1}^{n-1}\rho_i>0$, $x_i=\rho_i/(cM_i)$ for $i=1,\ldots,n-1$, and
$c=\sum_{i=1}^n\rho_i/M_i$ \cite[Lemma 7]{ChJu15}.
Finally, the result follows by combining the previous steps.
\end{proof}

\begin{lemma}\label{lem.B}
Let $\bm{w}\in H^1(\Omega;\R^{n-1})$. 
Then there exists a constant $C_B>0$, which only depends
on $D_{ij}$ and $M_i$, such that
$$
  \int_\Omega\na \bm{w}:\mathbb{B}(\bm{w})\na \bm{w} dz 
	\ge C_B\sum_{i=1}^n\int_\Omega|\na x_i^{1/2}|^2 dz,
$$
where ``:'' means summation over both matrix indices.
\end{lemma}

We recall that $\mathbb{B}(w)=\mathbb{A}(\bm{\rho}')h''(\bm{\rho}')^{-1}$ and
$h''$ is the Hessian of the entropy $h$ defined in \eqref{1.h}.
Lemma \ref{lem.B} is proved in \cite[Lemma 12]{ChJu15}.
It is shown in \cite[Lemma 9]{ChJu15} that $\mathbb{B}$ is symmetric and
positive definite.

\subsection{Solution to an approximate problem}
 
Let $T>0$, $M\in\N$, $\tau=T/M$, $k\in\{1,\ldots,M\}$, $\eps>0$, 
and $l\in\N$ with $l>d/2$. 
Then the embedding $H^l(\Omega)\hookrightarrow L^\infty(\Omega)$ is compact.
Given $\bm{w}^{k-1}\in L^\infty(\Omega;\R^{n-1})$, we wish to find 
$\bm{w}^k\in H^l(\Omega;\R^{n-1})$ such that
\begin{align}
  \frac{1}{\tau}\int_\Omega & \big(\bm{\rho}'(\bm{w}^k)-\bm{\rho}'(\bm{w}^{k-1})\big)
	\cdot \bm{q}dz
	+ \int_\Omega\na \bm{q}:\mathbb{B}(\bm{w}^k)\na \bm{w}^k dz \nonumber \\
	&{}+ \eps\int_\Omega\bigg(\sum_{|\bm{\alpha}|=l}D^{\bm{\alpha}} \bm{w}^k
	:D^{\bm{\alpha}} \bm{q} + \bm{w}^k\cdot \bm{q}\bigg)dz
	= \int_\Omega \bm{r}'(\bm{x}^k)\cdot \bm{q}dz, \label{2.approx}
\end{align}
for all $\bm{q}\in H^l(\Omega;\R^{n-1})$, where $\bm{r}'=(r_1,\ldots,r_{n-1})^\top$,
$x^k_i=\rho_i(\bm{w}^k)/(cM_i)$, and $\bm{\rho}'(\bm{w}^k)$ 
is defined in Lemma \ref{lem.inv}.
Moreover, $\bm{\alpha}=(\alpha_1,\ldots,\alpha_d)\in\N_0^d$ is a multi-index of order
$|\bm{\alpha}|=\alpha_1+\cdots+\alpha_d=l$ and 
$D^{\bm{\alpha}}=\pa^{|\bm{\alpha}|}/(\pa z_1^{\alpha_1}
\cdots$ $\pa z_d^{\alpha_d})$ is a partial derivative of order $l$.
The regularization with the $l$th-order derivative terms is needed since
the matrix $\mathbb{B}$ is not uniformly positive definite.
As $\bm{\rho}'$ is a bounded function of $\bm{w}$, 
we can apply the boundedness-by-entropy method
of \cite{Jue15} or \cite[Section 3.1]{ChJu15} to deduce the existence of a 
weak solution $\bm{w}^k\in H^l(\Omega;\R^{n-1})$ to \eqref{2.approx}. 

\subsection{Uniform estimates} 

The crucial step is to derive some a priori estimates.
The idea is to employ the test function 
$\bm{q}=\bm{w}^k$ in \eqref{2.approx} and to proceed
as in the proof of Lemma 14 of \cite{ChJu15}. The reaction terms have no influence
as the following lemma shows.

\begin{lemma}\label{lem.r}
It holds that
$$
  \bm{r}'(\bm{x}^k)\cdot \bm{w}^k = \sum_{i=1}^{n-1} r_i(\bm{x}^k)w_i^k\le 0.
$$
\end{lemma}

\begin{proof}
Let $\bm{x}=\bm{x}^k$ and $\bm{w}=\bm{w}^k$ to simplify. 
We deduce from \eqref{2.w} and total mass
conservation \eqref{1.ctm} that $\sum_{i=1}^{n-1}r_i(\bm{x})=-r_n(\bm{x})$ and
\begin{align}
  \bm{r}'(\bm{x})\cdot \bm{w}
	&= \sum_{i=1}^{n-1}r_i(\bm{x})\bigg(\frac{1}{M_i}\ln\frac{x_i}{x_{i\infty}} 
	- \frac{1}{M_n}\ln\frac{x_n}{x_{n\infty}}\bigg) \nonumber \\
	&= \sum_{i=1}^{n-1}\frac{r_i(\bm{x})}{M_i}\ln\frac{x_i}{x_{i\infty}} 
	- \frac{1}{M_n}\ln\frac{x_n}{x_{n\infty}}\sum_{i=1}^{n-1}r_i(\bm{x}) 
	= \sum_{i=1}^{n}\frac{r_i(\bm{x})}{M_i}\ln\frac{x_i}{x_{i\infty}}. \label{3.rr}
\end{align}
In view of definition \eqref{1.reac} of $r_i$ and $\bm{x}_\infty \in \E$, 
the last expression becomes
\begin{align*}
  \bm{r}'(\bm{x})\cdot \bm{w} &= \sum_{i=1}^n\sum_{a=1}^N(\beta_i^a-\alpha_i^a)
	(k_f^a \bm{x}^{\bm{\alpha}^a}-k_b^a \bm{x}^{\bm{\beta}^a})\ln\frac{x_i}{x_{i\infty}} \\
	&= \sum_{i=1}^n\sum_{a=1}^N	(k_f^a \bm{x}^{\bm{\alpha}^a}-k_b^a \bm{x}^{\bm{\beta}^a})
  \ln\frac{x_i^{\beta_i^a}x_{i\infty}^{\alpha_i^a}}{x_i^{\alpha_i^a}
	x_{i\infty}^{\beta_i^a}} \\
	&= \sum_{a=1}^N	(k_f^a \bm{x}^{\bm{\alpha}^a}-k_b^a \bm{x}^{\bm{\beta}^a})
  \ln\frac{\bm{x}^{\bm{\beta}^a}\bm{x}_{\infty}^{\bm{\alpha}^a}}{\bm{x}^{\bm{\alpha}^a}
	\bm{x}_{\infty}^{\bm{\beta}^a}} \\
	&= \sum_{a=1}^N	(k_f^a \bm{x}^{\bm{\alpha}^a}-k_b^a \bm{x}^{\bm{\beta}^a})
	\ln\frac{k_b^a \bm{x}^{\bm{\beta}^a}}{k_f^a \bm{x}^{\bm{\alpha}^a}} \le 0,
\end{align*}
because of the monotonicity of the logarithm.
\end{proof}

Taking into account Lemma \ref{lem.r},
the estimations of Section 3.2 in \cite{ChJu15} lead to the discrete entropy inequality
\begin{align}
  \int_\Omega h((\bm{\rho}')^k)dz 
	&+ C\tau\sum_{j=1}^k\sum_{i=1}^{n}\|\na (x_i^j)^{1/2}\|_{L^2(\Omega)}^2 
	+ \tau\sum_{j=1}^k\sum_{i=1}^n\int_\Omega (-r_i(\bm{x}^j)\cdot\bm{w}^j) dz \nonumber \\
	&{}+ \eps\tau\sum_{j=1}^k\sum_{i=1}^{n-1}\int_\Omega
	\bigg(\sum_{|\alpha|=l}(D^{\bm{\alpha}} w_i^j)^2 + (w_i^j)^2\bigg)dz
	\le \int_\Omega h((\bm{\rho}')^0_\eta)dz, \label{2.ei}
\end{align}
where $(\bm{\rho}')^0_\eta$ is the vector of strictly positive approximations
of the initial vector $(\bm{\rho}^0)'=(\rho_1^0,\ldots,\rho_{n-1}^0)^\top$ and
$C>0$ is a generic constant independent of $\tau$ and $\eps$.
This shows that
$$
  \tau\sum_{j=1}^k\|x_i^j\|_{H^1(\Omega)}^2 
	+ \eps\tau\sum_{j=1}^n\|w_i^j\|_{H^l(\Omega)}^2 \le C, \quad i=1,\ldots,n,
$$
where $C>0$ is independent of $\eps$ and $\tau$.
From these estimates and the boundedness of the reaction terms, we infer a uniform
bound for the discrete time derivative,
$$
  \tau\sum_{k=1}^M\sum_{i=1}^{n-1}
	\big\|\tau^{-1}(\rho_i^k-\rho_i^{k-1})\big\|_{H^l(\Omega)'}^2 \le C.
$$
These estimates are sufficient to perform the limit $\eps\to 0$ and $\tau\to 0$
in \eqref{2.approx} as in Section 3.3 of \cite{ChJu15} showing that the limit
satisfies \eqref{2.weak} and therefore is a global weak solution to 
\eqref{1.eq}--\eqref{1.reac}. 

\begin{remark}[Discrete entropy inequality]\label{rem.dei}\rm
Before summing from $j=1,\ldots,k$, we 
can formulate the discrete entropy inequality \eqref{2.ei} as
$$
  E[\bm{x}^k|\bm{x}_\infty] + \tau D[\bm{x}^k]
	+ C\eps\tau\sum_{i=1}^{n-1}\|w_i^k\|_{H^l(\Omega)}^2
	\le E[\bm{x}^{k-1}|\bm{x}_\infty].
$$
This estimate is the discrete analogue of \eqref{1.dEdt} and it will be needed
in the proof of part (iii) of Theorem \ref{thm.main}; see Section \ref{sec.proof}.
\qed
\end{remark}


\section{Convergence to equilibrium under detailed balance}\label{sec.time1}

In this section, we prove parts (ii) and (iii) of Theorem \ref{thm.main}.
First, we discuss the conservation laws and the existence of an equilibrium state.

\subsection{Conservation laws}\label{sec.cl}

We set $R_i=r_i/M_i$, $\bm{J}_i=\bm{j}_i/M_i$ and $\bm{R} = (R_1,\ldots,R_n)^\top$,
$\mathbb{J}=(\bm{J}_1,\ldots,\bm{J}_n)^\top$, 
$\bm{c}=(c_1,\ldots,c_n)^\top$, where we recall that
$c_i=\rho_i/M_i$. Dividing the $i$th-equation of \eqref{1.eq} by $M_i$, we can reformulate
them in vector form as
\begin{equation}\label{3.eq}
  \pa_t\bm{c} + \diver\mathbb{J} = \bm{R}.
\end{equation}
Let $\mathbb{W}=(\beta_i^a-\alpha_i^a)\in\R^{n\times N}$
be the Wegscheider matrix and let
$m=\dim\operatorname{ker}(\mathbb{W}^\top)$. Note that $m\ge 1$ since
it follows from the conservation of total mass, $\sum_{i=1}^n r_i(\bm{x})=0$, that
$\bm{M}^\top\mathbb{W}=0$, i.e.,
the vector $\bm{M}=(M_1,\ldots,M_n)^\top$ belongs to 
$\operatorname{ker}(\mathbb{W}^\top)$.
Let the row vectors $\bm{q}_1,\ldots,\bm{q}_m\in\R^{1\times n}$ be a basis of
the left null space of $\mathbb{W}$, i.e.\ $\bm{q}_i\mathbb{W}=0$ 
for $i=1,\ldots,m$. In particular, $\bm{q}_i^\top\in\operatorname{ker}(\mathbb{W}^\top)$.
Finally, let $\Q=(Q_{ij})\in\R^{m\times n}$ be the matrix
with rows $\bm{q}_j$. 

We claim that system \eqref{3.eq} (with no-flux boundary conditions)
possesses precisely $m$ linear independent 
conservation laws. 

\begin{lemma}[Conservation laws]\label{lem.cl}
Let $\bm{\rho}$ be a weak solution to \eqref{1.eq}-\eqref{1.reac} in the sense of 
Theorem \ref{thm.ex}. Then the following conservation laws hold:
$$
  \Q\overline{\bm{c}}(t) = \MM, \quad t>0,
$$
where $\MM=\Q\overline{\bm{c}}^0$ is called the initial mass vector and 
$c_i^0=\rho_i^0/M_i$, $i=1,\ldots,n$.
\end{lemma}

Note that, by changing the sign of the rows of $\Q$ if necessary, 
we can always choose $\Q$ such that $\MM$ is positive componentwise. 

\begin{proof}
We observe that the definitions of $\Q$ and $r_i(\bm{x}) = M_iR_i(\bm{x})$ in 
\eqref{1.reac} imply that $\Q\bm{R}=0$. 
Choosing $\bm{q}_j=(Q_{j1},\ldots,$ $Q_{jn})$ as a test function 
in the weak formulation of \eqref{3.eq} 
and observing that $\na\bm{q}_j=0$, we find that
$$
  \int_0^t\int_\Omega\pa_t(\Q\bm{c})_jdz ds 
	= \sum_{i=1}^n\int_0^t\int_\Omega \pa_t c_i Q_{ji}dz ds
	= \sum_{i=1}^n\int_0^t\int_\Omega R_iQ_{ji}dzds
	= \int_0^t\int_\Omega(\Q\bm{R})_jdzds = 0.
$$
This shows that
$$
  \int_\Omega \Q\bm{c}(t)dz = \int_\Omega \Q\bm{c}^0dz, \quad t>0,
$$
or $\Q\overline{\bm{c}}(t)=\Q\overline{\bm{c}}^0=:\MM$, where $c_i^0=\rho_i^0/M_i$
is the initial concentration. 
\end{proof}

\begin{lemma}\label{lem.zeta}
There exists a row vector $\bm{\zeta}\in\R^{1\times m}$ such that
$\bm{\zeta}\Q=\bm{M}^\top$ and $\bm{\zeta}\MM=1$.
\end{lemma}

\begin{proof}
Since $\bm{M}$ lies in the kernel of $\mathbb{W^\top}$ and the rows of
$\Q$ form a basis of this space, we have 
$\bm{M}\in\operatorname{ker}(\mathbb{W}^\top)=\operatorname{ran}(\Q^\top)$.
We infer that there exists a row vector $\bm{\zeta}\in\R^{1\times m}$ such that
$\Q^\top\bm{\zeta}^\top=\bm{M}$ or $\bm{\zeta}\Q=\bm{M}^\top$. Moreover, by recalling 
$|\Omega| = 1$ and $\sum_{i=1}^n\rho_i^0 = 1$ in $\Omega$,
$$
  1 = \int_\Omega \sum_{i=1}^{n}\rho_i^0dz =   \sum_{i=1}^n\overline{\rho_i}^0
	= \sum_{i=1}^n M_i\overline{c_i}^0 = \bm{M}^\top\overline{\bm{c}}^0 
	= \bm{\zeta}\Q\overline{\bm{c}}^0 = \bm{\zeta}\MM,
$$
using the definition of $\MM$ in Lemma \ref{lem.cl}.
\end{proof}


\subsection{Detailed-balanced condition}\label{sec.dbc}

The relative entropy \eqref{1.ent} is formally 
a Lyapunov functional along the trajectories
of \eqref{1.eq}-\eqref{1.reac} for $\bm{x}_\infty\in \E$. 
Note that $\E$ generally is a manifold of detailed-balanced equilibria.
To identify uniquely the detailed-balanced equilibrium, we need to take into account
the conservation laws. This subsection is concerned with the existence of 
a unique positive detailed-balanced equilibrium satisfying the conservation laws. 

For chemical reaction networks in the context of ordinary differential equations (ODE),
the existence of a unique equilibrium state was proved by Horn and Jackson
\cite{HoJa72}; also see \cite{Fei95}. The difficulty in this work lies in the fact that 
the reactions are modeled by molar fractions $\bm{x}$, 
while the conservation laws are presented by concentrations $\bm{c}$. Our idea is to 
enlarge the space $\R_+^n$ of concentrations $(c_1,\ldots,c_n)$ by adding the
total concentration $c = \sum_{i=1}^{n}c_i \in\R_+$, 
which is considered to be an independent variable,
and then to employ the ideas by Feinberg \cite{Fei95} to the augmented space $\R_+^{n+1}$.
To this end, let
\begin{equation}\label{3.omega}
  \bm{\omega} = (\omega_1,\ldots,\omega_{n+1}) = (c_1,\ldots,c_n,c),
\end{equation}
and define the vectors in $\R^{n+1}$
\begin{equation}\label{3.munu}
\begin{aligned}
  \bm{\mu}^a &= \bigg(\alpha_1^a,\ldots,\alpha_n^a,
	\bigg(\sum_{i=1}^n(\beta_i^a-\alpha_i^a)\bigg)^+\bigg), \\
	\bm{\nu}^a &= \bigg(\beta_1^a,\ldots,\beta_n^a,
	\bigg(\sum_{i=1}^n(\alpha_i^a-\beta_i^a)\bigg)^+\bigg),
\end{aligned}
\end{equation}
where $y^+=\max\{0,y\}$. Finally, we write $\mathbf{1}_n=(1,\ldots,1)^\top\in\R^n$ 
and $\mathbf{1}_{n+1}=(1,\ldots,1)^\top\in\R^{n+1}$. 
The main result of this subsection is the following.

\begin{theorem}[Existence of a unique detailed-balanced equilibrium]\label{thm.equi}
Assume that \eqref{A2} holds and let $\MM\in\R_+^m$ be an initial mass vector and 
$\bm{\zeta}\in\R^{1\times m}$ be a row vector such that $\bm{\zeta}\MM=1$. 
Then there exists a unique
positive detailed-balanced equilibrium $\bm{x}_\infty \in \E$ 
satisfying the conservation laws and the normalization condition \eqref{1.dbc}.
\end{theorem}

To prove Theorem \ref{thm.equi} we first show the existence of an ''equilibrium`` 
in the augmented space.

\begin{proposition}\label{pro.aug}
Suppose the assumptions of Theorem \ref{thm.equi} hold. 
Then there exists a unique $\bm{\omega}\in\R_+^{n+1}$ satisfying
\begin{equation}\label{3.equi}
  k_f^a\bm{\omega}^{\bm{\mu}^a} = k_b^a\bm{\omega}^{\bm{\nu}^a}, \quad a=1,\ldots,N,
	\quad \widehat\Q\bm{\omega} = \wMM,
\end{equation}
where $\widehat{\Q}$ and $\wMM$ are defined by
$$
  \widehat{\Q} = \begin{pmatrix}
	\Q & \bm{0} \\
	\bm{1}_n^\top & -1 
	\end{pmatrix}\in\R^{(m+1)\times(n+1)}, \quad
	\wMM = \begin{pmatrix}
	\MM \\ 0 \end{pmatrix}\in\R^{n+1}.
$$
\end{proposition}

Before proving this result, we first show that Theorem \ref{thm.equi}
follows from Proposition \ref{pro.aug}.

\begin{proof}[Proof of Theorem \ref{thm.equi}]
Let $\bm{\omega} = (c_{1\infty}, \ldots, c_{n\infty}, c_\infty)$ be the equilibrium 
in the augmented space constructed in Proposition \ref{pro.aug}. 
Define $x_{i\infty}=c_{i\infty}/c_\infty$. We will prove that 
$\bm{x}_\infty$ is an element of $\E$ and satisfies \eqref{1.dbc}. 
Indeed, for any $a = 1,\ldots, N$, let 
$\gamma^a:=\sum_{i=}^n(\alpha_i^a-\beta_i^a)$ and assume first that $\gamma^a\ge 0$. 
Then
$$
  k_f^a\prod_{i=1}^n c_{i\infty}^{\alpha_i^a}
  = k_f^a\bm{\omega}^{\bm{\mu}^a}
	= k_b^a\bm{\omega}^{\bm{\nu}^a}
	= k_b^a\prod_{i=1}^n c_{i\infty}^{\beta_i^a}c_\infty^{\gamma^a}
$$
is equivalent to
$$
  k_f^a\bm{x}_\infty^{\bm{\alpha}^a}
	= k_f^a\prod_{i=1}^n c_{i\infty}^{\alpha_i^a} c_\infty^{-\sum_{i=1}^n\alpha_i^a}
	= k_b^a\prod_{i=1}^n c_{i\infty}^{\beta_i^a} c_\infty^{-\sum_{i=1}^n\beta_i^a}
	= k_b^a\bm{x}_\infty^{\bm{\beta}^a}.
$$
The case $\gamma^a\le 0$ can be treated in an analogous way. 
Thus, $\bm{x}_\infty \in \E$. It follows immediately from 
$\widehat{\Q}\bm{\omega} = \widehat{\bm{M}}^0$ that $\Q \bm{c}_\infty = \MM$ and 
$\sum_{i=1}^{n}c_{i\infty} = c_\infty$. The latter identity implies that
$\sum_{i=1}^n x_{i\infty} = 1$ due to $x_{i\infty} = c_{i\infty}/c_\infty$. 
Therefore $\bm{x}_\infty$ satisfies \eqref{1.dbc}.
\end{proof}

The aim now is to prove Proposition \ref{pro.aug}. For this, 
we introduce the following definitions:
\begin{align*}
  X_1 &= \bigg\{\bm{\omega}\in\R_+^{n+1}:\;k_f^a\bm{\omega}^{\bm{\mu}^a} 
	= k_b^a\bm{\omega}^{\bm{\nu}^a}\mbox{ for }a=1,\ldots,N\bigg\}, \\
	X_2 &= \bigg\{\bm{\omega}\in\R_+^{n+1}:\; \widehat{\Q}\bm{\omega}=\wMM\bigg\}.
\end{align*}
We argue that $X_1$ and $X_2$ are not empty. Indeed, due to \eqref{A2}, 
there exists $\bm{x}_\infty \in \E$.  Fix any $\omega_{n+1,\infty}\in (0,\infty)$ and 
define $\omega_{i\infty} = x_{i\infty}\omega_{n+1,\infty}$ for all $i=1,\ldots, n$.
We obtain immediately $\bm{\omega}_\infty = (\omega_{1\infty},\ldots, 
\omega_{n+1,\infty})\in X_1$. Concerning $X_2$, we see that there exists 
$\bm{\omega}' = (\omega_1,\ldots, \omega_n) \in \R_+^n$ such that 
$\Q \bm{\omega}' = \MM$ since $\mathrm{rank}(\Q) = m < n$. By defining 
$\omega_{n+1} = \sum_{i=1}^n\omega_i$, 
we infer that $\bm{\omega} = (\bm{\omega}', \omega_{n+1})\in X_2$.

\begin{lemma}\label{lem.repres}
Let $\MM\in\R_+^m$ and $\bm{\zeta}\in\R^{1\times m}$ with $\bm{\zeta}\MM=1$,
let $\bm{\omega}_\infty\in X_1$ and $\bm{p}\in X_2$. Then the following statements
are equivalent:
\begin{itemize}[leftmargin=10mm]
\item There exists a unique vector $\bm{\omega}\in X_1\cap X_2$.
\item There exists a unique vector 
$\bm{\varphi}^*\in\operatorname{span}\{\bm{q}_1^\top,\ldots,\bm{q}_m^\top\}$ 
($\bm{q}_i$ is the $i$th row of $\Q$) and a unique number
$z_{m+1}\in\R$ such that
\begin{equation}\label{3.repres}
  \bm{\omega}'_\infty e^{\bm{\varphi}^*} - e^{-z_{m+1}}\bm{p}'\in\operatorname{ker}\Q, 
	\quad \langle e^{\bm{\varphi}^*}\bm{\omega}'_\infty,\mathbf{1}_n\rangle 
	= \omega_{n+1,\infty}.
\end{equation}
\end{itemize}
\end{lemma}

Here, we denote $\bm{p}'=(p_1,\ldots,p_{n})$ and $\bm{\omega}'_\infty e^{\bm{\varphi}^*}$
equals the vector with components $\omega_{i\infty} e^{\varphi_i^*}$,
$i=1,\ldots,n$. Observe that $\operatorname{span}\{\bm{q}_1^\top,\ldots,\bm{q}_m^\top\}
=\operatorname{ran}(\Q^\top)$.

\begin{proof}
We first claim that
\begin{align*}
  X_1 &= \bigg\{\bm{\omega}\in\R_+^{n+1}:\;\exists z_{m+1}\in\R,\ \bm{\varphi}^*\in
	\operatorname{ran}(\Q^\top)\; \text{ such that }\; \bm{\omega} 
	= e^{z_{m+1}}\begin{pmatrix} \bm{\omega}'_\infty e^{\bm{\varphi}^*} \\
	\omega_{n+1,\infty} \end{pmatrix}\bigg\}.
\end{align*}
Indeed, $\bm{\omega}\in X_1$ holds if and only if 
$\bm{\omega}_\infty^{\bm{\nu}^a-\bm{\mu}^a}=k_f^a/k_b^a
=\bm{\omega}^{\bm{\nu}^a-\bm{\mu}^a}$. Taking the logarithm componentwise, this
becomes
$$
  \langle\log\bm{\omega}_\infty,\bm{\nu}^a-\bm{\mu}^a\rangle
	= \langle\log\bm{\omega},\bm{\nu}^a-\bm{\mu}^a\rangle, \quad a=1,\ldots,N.
$$
This means that $\bm{\varphi}:=\log(\bm{\omega}/\bm{\omega}_\infty)
= \log\bm{\omega}-\log\bm{\omega}_\infty\in\operatorname{ker}
\{\bm{\nu}^a-\bm{\mu}^a\}_{a=1,\ldots,N}$. 
By definition of $\bm{\mu}^a$ and $\bm{\nu}^a$, we know that 
$$
  \operatorname{ker}\{\bm{\nu}^a-\bm{\mu}^a\}_{a=1,\ldots,N}
  = \operatorname{span}\big\{(\bm{q}_1^\top,0)^\top,\ldots,(\bm{q}_m^\top,0)^\top,
	\mathbf{1}_{n+1}\big\}.
$$
Thus, there exist numbers $z_1,\ldots,z_{m+1}\in\R$ such that
$$
  \bm{\varphi} = \sum_{i=1}^m z_i
	\begin{pmatrix} \bm{q}_i^\top \\ 0 \end{pmatrix} + z_{m+1}\mathbf{1}_{n+1}
	= \begin{pmatrix}\bm{\varphi}^* + z_{m+1}\mathbf{1}_n \\ z_{m+1}	\end{pmatrix},
$$
where $\bm{\varphi}^*=\sum_{i=1}^m z_i\bm{q}_i^\top \in {\rm ran}(\Q^\top)$. 
It follows from the definition of $\bm{\varphi}$ that
$$
  \frac{\bm{\omega}}{\bm{\omega}_\infty} = e^{\bm{\varphi}} 
	= \exp\begin{pmatrix}\bm{\varphi}^* + z_{m+1}\mathbf{1}_n \\
	z_{m+1}	\end{pmatrix}
	= e^{z_{m+1}}\begin{pmatrix} e^{\bm{\varphi}^*} \\ 1 \end{pmatrix}.
$$
We conclude that $\bm{\omega}\in X_1$ if and only if
$$
  \bm{\omega} = \bm{\omega}_\infty e^{z_{m+1}}
	\begin{pmatrix} e^{\bm{\varphi}^*} \\ 1 \end{pmatrix}
	= e^{z_{m+1}}\begin{pmatrix} \bm{\omega}'_\infty e^{\bm{\varphi}^*} \\
	\omega_{n+1,\infty} \end{pmatrix},
$$
and this proves the claim.

Next, fixing $\bm{p}\in X_2$, it holds that $\bm{\omega}\in X_2$ if and only if
\begin{align*}
  \bm{0} &= \widehat{\Q}(\bm{\omega}-\bm{p})
	= \begin{pmatrix}
	\Q & \bm{0} \\
	\bm{1}_n^\top & -1 
	\end{pmatrix}
	\begin{pmatrix} \bm{\omega}'-\bm{p}' \\ \omega_{n+1}-p_{n+1} \end{pmatrix} \\
	&= \begin{pmatrix}
	\Q(\bm{\omega}'-\bm{p}') \\ \langle\mathbf{1}_n,\bm{\omega}'-\bm{p}'\rangle 
	- (\omega_{n+1}-p_{n+1}) \end{pmatrix}.
\end{align*}
Consequently, in view of the preceding claim,
we have $\bm{\omega}\in X_1\cap X_2$ if and only if
$$
  \bm{0} = \widehat\Q(\bm{\omega}-\bm{p})
  = \begin{pmatrix}
	\Q(e^{z_{m+1}}\bm{\omega}'_\infty e^{\bm{\varphi}^*}-\bm{p}') \\
	\langle\mathbf{1}_n,e^{z_{m+1}}\bm{\omega}'_\infty e^{\bm{\varphi}^*}-\bm{p}'\rangle
	- (e^{z_{m+1}}\omega_{n+1,\infty}-p_{n+1}) \end{pmatrix}.
$$
The first $n$ rows mean that
$\bm{\omega}'_\infty e^{\bm{\varphi}*}-e^{-z_{m+1}}\bm{p}'\in\operatorname{ker}\Q$.
Since $\bm{p}\in X_2$ and consequently $p_{n+1}=\sum_{i=1}^n p_i
=\langle\mathbf{1}_n,\bm{p}'\rangle$,
the last row simplifies to
$$
  0 = e^{z_{m+1}}\big(\langle e^{\bm{\varphi}^*}\bm{\omega}'_\infty,\mathbf{1}_n\rangle
	- \omega_{n+1,\infty}\big).
$$
This shows \eqref{3.repres} and ends the proof.
\end{proof}

We need one more lemma.
\begin{lemma}\cite[Proposition B.1]{Fei95}\label{lemfein}
Let $U$ be a linear subspace of $\R^n$ and $a=(a_1,\ldots,a_n), b=(b_1,\ldots,b_n) \in \R_+^n$. 
There exists a unique element $\mu=(\mu_1,\ldots,\mu_n) \in U^{\perp}$ such that 
$$
		ae^{\mu} - b \in U,
$$
where $ae^\mu=(a_1 e^{\mu_1},\ldots,a_n e^{\mu_n})$.
\end{lemma}

\begin{proof}[Proof of Proposition \ref{pro.aug}.]
{\em Step 1: Existence.} 
First, fixing 
$\bm{\omega}_\infty\in X_1$ and $\bm{p}\in X_2$, we claim that there exist $z_{m+1}\in\R$
and $\bm{\varphi}^*\in\operatorname{ran}(\Q^\top)$
such that \eqref{3.repres} holds. We apply Lemma \ref{lemfein} 
with $U=\operatorname{ker}\Q$, $a=\bm{\omega}'_\infty$, and
$b=e^{-z_{m+1}}\bm{p}'$, yielding the existence of a unique vector
$\bm{\varphi}^*(z_{m+1})\in U^\perp=\operatorname{ran}(\Q^\top)$ 
such that
\begin{equation}\label{3.feinberg}
  \bm{\omega}'_\infty e^{\bm{\varphi}^*(z_{m+1})} 
	- e^{-z_{m+1}}\bm{p}'\in\operatorname{ker}\Q.
\end{equation}
It remains to show the second equation in \eqref{3.repres}, 
i.e.\ to show that there exists a number $z_{m+1}^*\in\R$ such that
$\langle e^{\bm{\varphi}^*(z^*_{m+1})}\bm{\omega}'_\infty,\mathbf{1}_n\rangle
=\omega_{n+1,\infty}$. Then we set $\bm{\varphi}^*:=\bm{\varphi}^*(z_{m+1}^*)$,
and \eqref{3.feinberg} yields the first equation in \eqref{3.repres}.

We know that $\bm{M}\in\operatorname{span}\{\bm{q}_1^\top,\ldots,\bm{q}_m^\top\}$.
Then \eqref{3.feinberg} implies that 
$$
  \big\langle\bm{\omega}'_\infty e^{\bm{\varphi}^*(z_{m+1})} 
	- e^{-z_{m+1}}\bm{p}',\bm{M}\big\rangle = 0 \quad\mbox{or}\quad
	\big\langle\bm{\omega}'_\infty e^{\bm{\varphi}^*(z_{m+1})},\bm{M}\big\rangle
	= e^{-z_{m+1}}\langle\bm{p}',\bm{M}\rangle > 0.
$$
We deduce that
$$
  \lim_{z_{m+1}\to+\infty}\langle\bm{\omega}'_\infty e^{\bm{\varphi}^*(z_{m+1})},
	\bm{M}\rangle = 0,
	\quad \lim_{z_{m+1}\to-\infty}\langle\bm{\omega}'_\infty e^{\bm{\varphi}^*(z_{m+1})}, 
	\bm{M}\rangle = \infty.
$$
Moreover, since
$$
  \frac{1}{M_{\rm max}}\langle\bm{\omega}'_\infty e^{\bm{\varphi}^*(z_{m+1})},
	\bm{M}\rangle
	\le \langle\bm{\omega}'_\infty e^{\bm{\varphi}^*(z_{m+1})},\mathbf{1}_n\rangle
	\le \frac{1}{M_{\rm min}}\langle\bm{\omega}'_\infty e^{\bm{\varphi}^*(z_{m+1})},
	\bm{M}\rangle,
$$
it holds that
$$
  \lim_{z_{m+1}\to+\infty}\langle\bm{\omega}'_\infty e^{\bm{\varphi}^*(z_{m+1})},
	\mathbf{1}_n\rangle = 0,
	\quad \lim_{z_{m+1}\to-\infty}\langle\bm{\omega}'_\infty e^{\bm{\varphi}^*(z_{m+1})}, 
	\mathbf{1}_n\rangle = \infty.
$$
By continuity, there exists $z_{m+1}^*\in\R$ such that
$\langle e^{\bm{\varphi}^*(z^*_{m+1})}\bm{\omega}'_\infty,\mathbf{1}_n\rangle
=\omega_{n+1,\infty}$.

{\em Step 2: Uniqueness.} Assume that there exist 
$(\widehat{\bm{\varphi}},\widehat{z})$ and $(\widecheck{\bm{\varphi}},\widecheck{z})$ with
$\widehat{\bm{\varphi}}$, $\widecheck{\bm{\varphi}}\in\operatorname{ran}(\Q^\top)$ 
and $\widehat{z}$, $\widecheck{z}\in\R$ such that
\begin{align}
  & \bm{\omega}'_\infty e^{\widehat{\bm{\varphi}}} 
	- e^{-\widehat{z}}\bm{p}',\
	\bm{\omega}'_\infty e^{\widecheck{\bm{\varphi}}} 
	- e^{-\widecheck{z}}\bm{p}'\in\operatorname{ker}\Q, \label{3.uni1} \\
  & \langle \bm{\omega}'_\infty e^{\widehat{\bm{\varphi}}},\mathbf{1}_n\rangle
	= \omega_{n+1,\infty} 
	= \langle\bm{\omega}'_\infty e^{\widecheck{\bm{\varphi}}},\mathbf{1}_n\rangle.
  \label{3.uni2}
\end{align}
From \eqref{3.uni1} it follows that
$$
  e^{\widehat{z}}\bm{\omega}'_\infty e^{\widehat{\bm{\varphi}}} 
	- e^{\widecheck{z}}\bm{\omega}'_\infty e^{\widecheck{\bm{\varphi}}} 
	\in \operatorname{ker}\Q.
$$
We infer from $\widehat{\bm{\varphi}}-\widecheck{\bm{\varphi}}\in
\operatorname{ran}(\Q^\top)=\operatorname{span}\{\bm{q}_1^\top,\ldots,\bm{q}_m^\top\}$ 
that
\begin{align*}
  0 &= \big\langle e^{\widehat{z}}\bm{\omega}'_\infty e^{\widehat{\bm{\varphi}}} 
	- e^{\widecheck{z}}\bm{\omega}'_\infty e^{\widecheck{\bm{\varphi}}},
	\widehat{\bm{\varphi}}-\widecheck{\bm{\varphi}}\big\rangle \\
	&= e^{\widecheck{z}}\big\langle\bm{\omega}'_\infty
	(e^{\widehat{\bm{\varphi}}}-e^{\widecheck{\bm{\varphi}}}),
	(\widehat{\bm{\varphi}}-\widecheck{\bm{\varphi}})\big\rangle
	+ (e^{\widehat{z}}-e^{\widecheck{z}})
	\big\langle\bm{\omega}'_\infty e^{\widehat{\bm{\varphi}}},
	\widehat{\bm{\varphi}}-\widecheck{\bm{\varphi}}\big\rangle
	=: I_1 + I_2.
\end{align*}
Hence, we have $I_2=-I_1$ and because of
$$
  I_1 = e^{\widecheck{z}}\sum_{i=1}^n \omega_{i\infty}\big(e^{\widehat{\varphi}_i}
	- e^{\widecheck{\varphi}_i}\big)(\widehat{\varphi}_i-\widecheck{\varphi}_i)
	\ge 0,
$$
it holds that $I_2=-I_1\le 0$.

Now, if $\widehat{z}=\widecheck{z}$, Lemma \ref{lemfein} shows that
$\widehat{\bm{\varphi}}=\widecheck{\bm{\varphi}}$, and the proof is finished.
Thus, let us assume, without loss of generality, that $\widehat{z}>\widecheck{z}$.
Then the definition and nonpositivity of $I_2$ imply that
\begin{equation}\label{3.leq}
  \langle\bm{\omega}_\infty' e^{\widehat{\bm{\varphi}}},
	\widehat{\bm{\varphi}}-\widecheck{\bm{\varphi}}\rangle \le 0.
\end{equation}
Consider the function $f:\R^n\to\R$, $f(\bm{\varphi})=\sum_{i=1}^n \omega_{i\infty}
e^{\varphi_i}$. Then $\mathrm{D}f(\bm{\varphi})=\bm{\omega}'_\infty e^{\bm{\varphi}}$
and $\mathrm{D}^2f(\bm{\varphi})=\operatorname{diag}(\omega_{i\infty} 
e^{\varphi_i})_{i=1,\ldots,n}$ and so, $f$ is strictly convex.
Hence, by \eqref{3.uni2},
\begin{align*}
  \langle\bm{\omega}_\infty' e^{\widehat{\bm{\varphi}}},
	\widehat{\bm{\varphi}}-\widecheck{\bm{\varphi}}\rangle
	&= \langle\mathrm{D} f(\widehat{\bm{\varphi}}),
	\widehat{\bm{\varphi}}-\widecheck{\bm{\varphi}}\rangle
	\ge f(\widehat{\bm{\varphi}}) - f(\widecheck{\bm{\varphi}}) \\
	&= \langle \bm{\omega}'_\infty e^{\widehat{\bm{\varphi}}},\mathbf{1}_n\rangle
	- \langle \bm{\omega}'_\infty e^{\widecheck{\bm{\varphi}}},\mathbf{1}_n\rangle
  = 0.
\end{align*}
We deduce from this identity and \eqref{3.leq} that 
$\langle\bm{\omega}_\infty' e^{\widehat{\bm{\varphi}}},
\widehat{\bm{\varphi}}-\widecheck{\bm{\varphi}}\rangle=0$ and consequently,
$I_2=0$ and $I_1=-I_2=0$. By the monotonicity of the exponential function,
we infer that $\widehat{\bm{\varphi}}=\widecheck{\bm{\varphi}}$.
Then, taking the difference of the two vectors in \eqref{3.uni1}, we have
$(e^{-\widehat{z}}-e^{-\widecheck{z}})\bm{p}'\in\operatorname{ker}\Q$. 
Since $\widehat{z}\neq \widecheck{z}$, this shows that 
$\bm{p}'\in\operatorname{ker}\Q$ and therefore
$\Q\bm{p}'=\bm{0}$ contradicting the fact that $\bm{p}\in X_2$ and in particular
$\Q\bm{p}'=\MM\neq \bm{0}$. Thus, $\widehat{z}$ and $\widecheck{z}$ must
coincide, and uniqueness holds.
\end{proof}

\begin{remark}[Boundary equilibria and Global Attractor Conjecture]\label{rem.be}\rm
\sloppy
Besides the unique positive detailed-balanced equilibrium obtained in Theorem 
\ref{thm.equi}, there might exist (possibly infinitely many) boundary equilibria
$\bm{x}^*\in\pa\E$. The convergence of solutions
to reaction systems towards the positive
equilibrium under the presence of boundary equilibria is
a subtle problem, even in the ODE setting. The main reason is that if a
trajectory converges to a boundary equilibrium, the entropy production $D[\bm{x}]$
vanishes while the relative entropy $E[\bm{x}|\bm{x}_\infty]$ remains positive,
which means that the entropy-production inequality \eqref{1.eep} is not true
in general. However, it is conjectured, still in the ODE setting, that the
positive detailed-balanced equilibrium is the only attracting point despite the
presence of boundary equilibria. This is called the {\em Global Attractor Conjecture},
and it is considered as one of the most important problems in chemical reaction network
theory; see, e.g., \cite{And11,GMS14} for partial answers. Recently, a full
proof of this conjecture in the ODE setting has been proposed in \cite{Cra15},
but the result is still under verification. See also \cite{DFT17,FeTa17a} for 
reaction-diffusion systems possessing boundary equilibria.
\qed
\end{remark}

%
%


\subsection{Preliminary estimates for the entropy and entropy production}\label{sec.prem}

We derive some estimates for the relative entropy \eqref{1.ent}
and the entropy production \eqref{1.ep} from below and above.
In the following, let $\rho_1,\ldots,\rho_n:\Omega\to[0,\infty)$ be integrable functions
such that $\sum_{i=1}^n\rho_i=1$ in $\Omega$
and set $c_i=\rho_i/M_i$ and $x_i=c_i/c$ for $i=1,\ldots,n$. 
We assume that the functions have the same regularity as the weak solutions
from Theorem \ref{thm.ex}.
For later reference, we note the following inequalities, which give bounds
on the total concentration only depending on the molar masses:
\begin{equation}\label{3.c}
  \frac{1}{M_{\rm max}} \le c = \sum_{i=1}^n\frac{\rho_i}{M_i} \le \frac{1}{M_{\rm min}}
	\quad\mbox{in }\Omega,
\end{equation}
where $M_{\rm max}=\max_{i=1,\ldots,n}M_i$ and $M_{\rm min}=\min_{i=1,\ldots,n}M_i$.
Moreover, given the unique equilibrium $\bm{x}_\infty$ according to 
Theorem \ref{thm.equi}, we observe that
$\sum_{i=1}^n\rho_{i\infty}/M_i=\sum_{i=1}^n c_{i\infty}
=c_\infty\sum_{i=1}^n x_{i\infty}=c_\infty$, and consequently,
\begin{equation}\label{3.cinfty}
   \frac{1}{M_{\rm max}} \le c_\infty \le \frac{1}{M_{\rm min}}.
\end{equation}

\begin{lemma}\label{lem.E.le}
There exists a constant $C>0$, only depending on $M_{\rm min}$, $M_{\rm max}$, 
and $\bm{x}_\infty$, such that
$$
  E[\bm{x}|\bm{x}_\infty] \le C\sum_{i=1}^n\bigg(\int_\Omega
	\Big(c_i^{1/2}-\overline{c_i^{1/2}}\,\Big)^2 dz
	+ \big(\overline{c_i}^{1/2}-c_{i\infty}^{1/2}\big)^2\bigg).
$$
\end{lemma}

\begin{proof}
We use $\sum_{i=1}^n x_i=\sum_{i=1}^n x_{i\infty}=1$ to reformulate the
relative entropy
\begin{align*}
  E[\bm{x}|\bm{x}_\infty] 
	&= \sum_{i=1}^n\int_\Omega c\bigg(x_i\ln\frac{x_i}{x_{i\infty}} - x_i + x_{i\infty}
	\bigg)dz \\
	&= \sum_{i=1}^n\int_\Omega cx_{i\infty}
	\bigg(\frac{x_i}{x_{i\infty}}\ln\frac{x_i}{x_{i\infty}} - \frac{x_i}{x_{i\infty}}
	+ 1\bigg)dz.
\end{align*}
The function $\Phi(y)=(y\ln y-y+1)/(y^{1/2}-1)^2$ is continuous and 
nondecreasing on $\R_+$. Therefore,
using \eqref{3.c},
\begin{align}
  E[\bm{x}|\bm{x}_\infty] 
	&= \sum_{i=1}^n\int_\Omega cx_{i\infty}
	\Phi\bigg(\frac{x_i}{x_{i\infty}}\bigg)\bigg(\bigg(\frac{x_i}{x_{i\infty}}\bigg)^{1/2}
	-1\bigg)^2 dz \nonumber \\
	&\le \frac{1}{M_{\rm min}}\sum_{i=1}^n\Phi\bigg(\frac{1}{x_{i\infty}}\bigg)\frac{1}{x_{i\infty}}
	\int_\Omega(x_i-x_{i\infty})^2 dz
	\le C\sum_{i=1}^n\int_\Omega(x_i-x_{i\infty})^2 dz \label{3.aux1}
\end{align}
for some constant $C>0$ only depending on $M_{\rm min}$ and $\bm{x}_\infty$.

It remains to formulate the square on the right-hand side
in terms of the partial concentrations.
To this end, we set $f_i(\bm{c}) = c_i/c$ 
for $\bm{c}=(c_1,\ldots,c_n)$ and $c=\sum_{j=1}^n c_j$.
By definition of the molar fractions $x_i$ and $x_{i\infty}$, we have
$x_i=f_i(\bm{c})$ and $x_{i\infty}=f_i(\bm{c}_\infty)$. The estimates
$$
 \left|\frac{\pa f_i}{\pa c_j}(\bm{c})\right| 
  \le \frac{1}{c}
	\le M_{\rm max}, \quad
	\left|\frac{\pa f_i}{\pa c_j}(\bm{c}_\infty)\right| \le \frac{1}{c_\infty} 
	\le M_{\rm max}
$$
imply that, for some $\bm{\xi}$ on the line between $\bm{c}$ and $\bm{c}_\infty$,
\begin{align*}
  \int_\Omega(x_i-x_{i\infty})^2 dz
	&= \int_\Omega(f_i(\bm{c})-f_i(\bm{c}_\infty))^2dz
	= \sum_{j=1}^n\int_\Omega\bigg(\frac{\pa f_i}{\pa c_j}(\bm{\xi})\bigg)^2
	(c_j-c_{j\infty})^2 dz \\
	&\le M_{\rm max}^2\sum_{j=1}^n\int_\Omega
	\big(c_j^{1/2}+c_{j\infty}^{1/2}\big)^2\big(c_j^{1/2}-c_{j\infty}^{1/2}\big)^2 dz \\
	&\le M_{\rm max}^2\bigg(\frac{2}{M_{\rm min}^{1/2}}\bigg)^2\sum_{i=1}^n\int_\Omega
	\big(c_i^{1/2}-c_{i\infty}^{1/2}\big)^2 dz \\
	&\le C\sum_{i=1}^n\int_\Omega\big(c_i^{1/2}-c_{i\infty}^{1/2}\big)^2 dz,
\end{align*}
and $C>0$ depends only on $M_{\rm min}$, $M_{\rm max}$, and $\bm{x}_\infty$.
Combining this estimate with \eqref{3.aux1} leads to
(here, we use that $|\Omega|=1$)
\begin{align}
  E[\bm{x}|\bm{x}_\infty] 
	&\le C\sum_{i=1}^n\int_\Omega\big(c_i^{1/2}-c_{i\infty}^{1/2}\big)^2 dz \nonumber \\
	&\le 2C\sum_{i=1}^n\bigg(\int_\Omega\Big(c_i^{1/2}-\overline{c_{i}^{1/2}}\,\Big)^2 dz
	+ \Big(\,\overline{c_{i}^{1/2}}-c_{i\infty}^{1/2}\Big)^2\bigg) \label{3.aux2} \\
	&\le 2C\sum_{i=1}^n\bigg(\int_\Omega\Big(c_i^{1/2}-\overline{c_{i}^{1/2}}\,\Big)^2 dz
	+ 2\Big(\,\overline{c_{i}^{1/2}}-\overline{c_i}^{1/2}\Big)^2
	+ 2\big(\overline{c_i}^{1/2}-c_{i\infty}^{1/2}\big)^2\bigg). \nonumber
\end{align}
We wish to estimate the second term. The Cauchy--Schwarz inequality gives 
$\overline{c_i^{1/2}}\le\overline{c_i}^{1/2}$ and hence
\begin{align*}
  \Big(\,\overline{c_{i}^{1/2}}-\overline{c_i}^{1/2}\Big)^2
	&= \Big(\overline{c_{i}^{1/2}}\Big)^2 + \overline{c_i}
	- 2\overline{c_{i}^{1/2}}\overline{c_i}^{1/2} \\
  &\le \Big(\overline{c_{i}^{1/2}}\Big)^2 + \overline{c_i}
	- 2\overline{c_{i}^{1/2}}\,\overline{c_i^{1/2}} 
	= \int_\Omega\Big(c_i^{1/2}-\overline{c_i^{1/2}}\,\Big)^2 dz.
\end{align*}
Putting this into \eqref{3.aux2}, it follows that
$$
  E[\bm{x}|\bm{x}_\infty] 
	\le 2C\sum_{i=1}^n\bigg(3\int_\Omega\Big(c_i^{1/2}-\overline{c_{i}^{1/2}}\,\Big)^2 dz
	+ 2\big(\overline{c_i}^{1/2}-c_{i\infty}^{1/2}\big)^2\bigg),
$$
and we conclude the proof.
\end{proof}

\begin{lemma}\label{lem.D.ge}
There exists a constant $C>0$, only depending on
$M_{\rm min}$ and $M_{\rm max}$, such that
$$
  D[\bm{x}] \ge C\left[\sum_{i=1}^n\int_\Omega|\na c_i^{1/2}|^2 dz
	+ \int_\Omega|\na c^{1/2}|^2 dz
	+ \sum_{a=1}^N\int_\Omega
	\big(k_f^a\bm{x}^{\bm{\alpha}^a}-k_b^a\bm{x}^{\bm{\beta}^a}\big)
	\ln\frac{k_f^a\bm{x}^{\bm{\alpha}^a}}{k_b^a\bm{x}^{\bm{\beta}^a}}dz\right].
$$
\end{lemma}

\begin{proof}
Lemma \ref{lem.B} shows that the first term in $D[\bm{x}]$ can be estimated from below:
$$
  \int_\Omega\na \bm{w}:\mathbb{B}(\bm{w})\na \bm{w} dz 
	\ge C_B\sum_{i=1}^n\int_\Omega|\na x_i^{1/2}|^2 dz.
$$
We claim that we can relate $\sum_{i=1}^n|\na x_i^{1/2}|^2$ and $|\na c^{1/2}|^2$.
For this, we proceed as in \cite[page 494]{ChJu15}. We infer from the definition
$x_i=c_i/c$ that $c\sum_{i=1}^n M_ix_i=\sum_{i=1}^n M_ic_i=\sum_{i=1}^n\rho_i=1$.
Therefore, inserting $c=1/\sum_{i=1}^n M_ix_i$ and using the Cauchy--Schwarz inequality,
\begin{align}
  |\na c^{1/2}|^2
	&= \frac{1}{4c}|\na c|^2
	= \frac{1}{4c}\bigg|\frac{-\sum_{i=1}^nM_i\na x_i}{(\sum_{i=1}^n M_ix_i)^2}\bigg|^2
	= c^3\bigg|\sum_{i=1}^n M_i x_i^{1/2}\na x_i^{1/2}\bigg|^2 \nonumber \\
	&\le nc^3\sum_{i=1}^n M_i^2 x_i|\na x_i^{1/2}|^2
	\le \frac{nM_{\rm max}^2}{M_{\rm min}^3}\sum_{i=1}^n|\na x_i^{1/2}|^2, \label{3.aux3}
\end{align}
where we used $c\le 1/M_{\rm min}$ (see \eqref{3.c}). 
Similarly, employing \eqref{3.aux3},
\begin{align}
  \sum_{i=1}^n|\na c_i^{1/2}|^2 &= \sum_{i=1}^n|\na(cx_i)^{1/2}|^2
	\le 2\sum_{i=1}^n x_i|\na c^{1/2}|^2 + 2\sum_{i=1}^n c|\na x_i^{1/2}|^2 \nonumber \\
	&= 2|\na c^{1/2}|^2 + 2c\sum_{i=1}^n |\na x_i^{1/2}|^2
	\le C\sum_{i=1}^n |\na x_i^{1/2}|^2, \label{3.aux4}
\end{align}
where $C>0$ depends only on $M_{\rm min}$ and $M_{\rm max}$.
Adding \eqref{3.aux3} and \eqref{3.aux4} and integrating over $\Omega$ 
then shows that, for another constant $C>0$,
$$
  \sum_{i=1}^n\int_\Omega |\na x_i^{1/2}|^2 
	\ge C\bigg(\sum_{i=1}^n\int_\Omega|\na c_i^{1/2}|^2 dz + \int_\Omega |\na c^{1/2}|^2 dz\bigg).
$$
The lemma then follows from definition \eqref{1.ep} of $D[\bm{x}]$.
\end{proof}

\begin{lemma}\label{lem.CKP}
There exists a constant $C_{\rm CKP}>0$, 
only depending on $M_{\rm max}$, such that
$$
  E[\bm{x}|\bm{x}_\infty] \ge C_{\rm CKP}\sum_{i=1}^n
	\|x_i-x_{i\infty}\|_{L^1(\Omega)}^2.
$$
\end{lemma}

\begin{proof}
The estimate is a consequence of the Csisz\'ar--Kullback--Pinsker inequality.
Since we are interested in the constant, we provide the (short) proof.
We recall that $1/M_{\rm max}\le c\le 1/M_{\rm min}$. 
Arguing as in \eqref{3.aux1} and using $\Phi(y)\ge 1$ for $y\in \R_+$, we obtain
\begin{align*}
  E[\bm{x}|\bm{x}_\infty]
	&= \sum_{i=1}^n\int_\Omega cx_{i\infty}\bigg(\frac{x_i}{x_{i\infty}}
	\ln\frac{x_i}{x_{i\infty}} - \frac{x_i}{x_{i\infty}} + 1\bigg)dz \\
	&= \sum_{i=1}^n\int_\Omega cx_{i\infty}\Phi\bigg(\frac{x_i}{x_{i\infty}}\bigg)
	\bigg(\bigg(\frac{x_i}{x_{i\infty}}\bigg)^{1/2}-1\bigg)^2 dz \\
	&\ge \frac{1}{M_{\rm max}}\sum_{i=1}^n\int_\Omega
	(x_i^{1/2}-x_{i\infty}^{1/2})^2 dz.
\end{align*}
Then, by the Cauchy--Schwarz inequality and the bounds $x_i\le 1$, $x_{i\infty}\le 1$,
\begin{align*}
  E[\bm{x}|\bm{x}_\infty]
  &\ge \frac{1}{M_{\rm max}}\sum_{i=1}^n\bigg(\int_\Omega
	|x_i^{1/2}-x_{i\infty}^{1/2}| dz\bigg)^2 \\
	&= \frac{1}{M_{\rm max}}\sum_{i=1}^n\bigg(\int_\Omega
	\frac{|x_i-x_{i\infty}|}{x_i^{1/2}+x_{i\infty}^{1/2}}dz\bigg)^2 \\
	&\ge \frac{1}{4M_{\rm max}}\sum_{i=1}^n\bigg(\int_\Omega|x_i-x_{i\infty}|dz\bigg)^2.
\end{align*}
This finishes the proof.
\end{proof}


\subsection{The case of equal homogeneities}\label{sec.eqhom}

The aim of this and the following subsection is the proof of the functional
inequality $D[\bm{x}]\ge\lambda E[\bm{x}|\bm{x}_\infty]$ for some $\lambda>0$.
For this, we will distinguish two cases, the case which we call 
\textit{equal homogeneities},
\begin{equation}\label{3.eqhom}
  \sum_{i=1}^n\alpha_i^a = \sum_{i=1}^n\beta_i^a \quad\mbox{for all }a=1,\ldots,N,
\end{equation}
and the case of \textit{unequal homogeneities}:
There exists $a \in \{1, \ldots, N\}$ such that
\begin{equation}\label{3.uneqhom}
  \sum_{i=1}^n\alpha_i^a \neq \sum_{i=1}^n\beta_i^a.
\end{equation}
This subsection is concerned with the first case. 

\begin{proposition}[Entropy entropy-production inequality; case of
equal homogeneities]\label{prop.eqhom}\mbox{ }\\
Fix $\MM \in \R_+^m$ such that $\bm{\zeta}\MM = 1$. Let $\bm{x}_\infty$ be the 
equilibrium constructed in Theorem \ref{thm.equi}. Assume that \eqref{3.eqhom} holds 
and system \eqref{1.eq}--\eqref{1.reac} has no boundary equilibria. Then there exists 
a constant $\lambda>0$, 
which is constructive up to a finite-dimensional inequality, such that
$$
  D[\bm{x}]\ge\lambda E[\bm{x}|\bm{x}_\infty] 
$$
for all functions $\bm{x}: \Omega \to \R_+^n$ having the same regularity 
as the corresponding solutions
in Theorem \ref{thm.ex}, and satisfying $\Q\overline{\bm{c}} = \MM$.
\end{proposition}

\begin{proof}
We use Lemma \ref{lem.E.le} and the Poincar\'e inequality to obtain
\begin{align*}
  E[\bm{x}|\bm{x}_\infty] 
	&\le C\sum_{i=1}^n\bigg(\int_\Omega
	\Big(c_i^{1/2}-\overline{c_i^{1/2}}\,\Big)^2 dz
	+ \big(\overline{c_i}^{1/2}-c_{i\infty}^{1/2}\big)^2\bigg) \\
	&\le C\sum_{i=1}^n\bigg\{\int_\Omega|\na c_i^{1/2}|^2 dz
	+ \bigg(\bigg(\frac{\overline{c_i}}{c_{i\infty}}\bigg)^{1/2}-1\bigg)^2\bigg\}.
\end{align*}
Next, we take into account estimate \cite[formula (11)]{FeTa17a} and  \cite[Lemma 2.7]{FeTa17a}:
\begin{align}
  E[\bm{x}|\bm{x}_\infty] 
	&\le C\sum_{i=1}^n\int_\Omega|\na c_i^{1/2}|^2 dz
	+ \frac{C}{H_1}\sum_{a=1}^N\bigg\{\bigg(\sqrt{\frac{\overline{\bm{c}}}{\bm{c}_\infty}}
	\bigg)^{\bm{\alpha^a}} - \bigg(\sqrt{\frac{\overline{\bm{c}}}{\bm{c}_\infty}}
	\bigg)^{\bm{\beta^a}}\bigg\}^2 \nonumber \\
  &\le C\sum_{i=1}^n\int_\Omega|\na c_i^{1/2}|^2 dz
	+ C\sum_{a=1}^N\big(k_f^a\bm{c}^{\bm{\alpha}^a}-k_b^a\bm{c}^{\bm{\beta}^a}\big)
	\ln\frac{k_f^a\bm{c}^{\bm{\alpha}^a}}{k_b^a\bm{c}^{\bm{\beta}^a}}, \label{3.aux5}
\end{align}
where $H_1>0$ is the constant in the finite-dimensional inequality (11) of \cite{FeTa17a}.
Observe that we can apply the results \cite{FeTa17a} since $\Q\overline{\bm{c}}=\MM$
is satisfied; see Lemma \ref{lem.cl}.

We claim that the last term is smaller or equal $D[\bm{x}]$. Indeed,
inserting the expression $x_i=c_i/c$ in the last term of the entropy production
\eqref{1.ep} and employing assumption \eqref{3.eqhom}, it follows that
\begin{align}
  \sum_{a=1}^N\int_\Omega(k_f^a\bm{x}^{\bm{\alpha}^a}-k_b^a\bm{x}^{\bm{\beta}^a})
	\ln\frac{k_f^a \bm{x}^{\bm{\alpha}^a}}{k_b^a \bm{x}^{\bm{\beta}^a}}dz
	&= \sum_{a=1}^N\int_\Omega\frac{1}{c^{\alpha_1^a+\cdots\alpha_n^a}}
	(k_f^a\bm{c}^{\bm{\alpha}^a}-k_b^a\bm{c}^{\bm{\beta}^a})
	\ln\frac{k_f^a \bm{c}^{\bm{\alpha}^a}}{k_b^a \bm{c}^{\bm{\beta}^a}}dz \nonumber \\
	&\ge C\sum_{a=1}^N\int_\Omega
	(k_f^a\bm{c}^{\bm{\alpha}^a}-k_b^a\bm{c}^{\bm{\beta}^a})
	\ln\frac{k_f^a \bm{c}^{\bm{\alpha}^a}}{k_b^a \bm{c}^{\bm{\beta}^a}}dz, \label{3.hom}
\end{align}
where we used in the last step $M_{\rm min}\le 1/c\le M_{\rm max}$. 
By Lemma \ref{lem.D.ge}, this shows that
$$
  D[\bm{x}] \ge C\sum_{i=1}^n\int_\Omega|\na c_i^{1/2}|^2 dz
	+ C\sum_{a=1}^N\int_\Omega
	(k_f^a\bm{c}^{\bm{\alpha}^a}-k_b^a\bm{c}^{\bm{\beta}^a})
	\ln\frac{k_f^a \bm{c}^{\bm{\alpha}^a}}{k_b^a \bm{c}^{\bm{\beta}^a}}dz,
$$
and combining this estimate with \eqref{3.aux5} concludes the proof.
\end{proof}


\subsection{The case of unequal homogeneities}\label{sec.uneqhom}

In this subsection, we consider the case \eqref{3.uneqhom} of unequal homogeneities.
Since we cannot replace $\bm{x}$ easily by $\bm{c}$ as in \eqref{3.hom}, the
estimates are much more involved than in the case of equal homogeneities.
Similar as in Section \ref{sec.dbc}, our idea is to introduce $c$ 
as a new variable and to lift the problem from the $n$ variables $c_1,\ldots,c_n$ 
to the $n+1$ variables $c_1,\ldots,c_n,c$.
Then $D[\bm{x}]$ is represented by $n+1$ variables $c_1,\ldots,c_n,c$
under the conservation laws $\Q\overline{\bm{c}}=\MM$ and the additional
constraint $c=\sum_{i=1}^n c_i$ and thus $\overline{c}=\sum_{i=1}^n \overline{c_i}$.
We employ the notation \eqref{3.omega} and \eqref{3.munu}.

First, let $\gamma^a:=\sum_{i=1}^n(\alpha_i^a-\beta_i^a)$ and assume that 
$\gamma^a\ge 0$. With the
definitions $x_i=c_i/c$, $\omega_i=c_i$ for $i=1,\ldots,n$, and $\omega_{n+1}=c$,
we compute
\begin{align*}
  \sum_{a=1}^N & \int_\Omega(k_f^a\bm{x}^{\bm{\alpha}^a}-k_b^a\bm{x}^{\bm{\beta}^a})
	\ln\frac{k_f^a \bm{x}^{\bm{\alpha}^a}}{k_b^a \bm{x}^{\bm{\beta}^a}}dz \\
	&= \sum_{a=1}^N\int_\Omega\bigg\{k_f^a\prod_{i=1}^n\bigg(\frac{c_i}{c}
	\bigg)^{\alpha_i^a}
	- k_b^a\prod_{i=1}^n\bigg(\frac{c_i}{c}\bigg)^{\beta_i^a}\bigg\}
	\ln\frac{k_f^a\prod_{i=1}^n(c_i/c)^{\alpha_i^a}}{k_b^a\prod_{i=1}^n
	(c_i/c)^{\beta_i^a}}dz
	\\
	&= \sum_{a=1}^N\int_\Omega\frac{1}{c^{\sum_{i=1}^n\alpha_i^a}}
	\bigg(k_f^a\prod_{i=1}^n c_i^{\alpha_i^a}
	- k_b^ac^{\gamma^a}\prod_{i=1}^n c_i^{\beta_i}\bigg)
	\ln\frac{k_f^a\prod_{i=1}^n c_i^{\alpha_i^a}}{k_b^ac^{\gamma^a}
	\prod_{i=1}^n c_i^{\beta_i}} dz \\
	&= \sum_{a=1}^N\int_\Omega\frac{1}{c^{\sum_{i=1}^n\alpha_i^a}}
	\big(k_f^a\bm{\omega}^{\bm{\mu}^a} - k_b^a\bm{\omega}^{\bm{\nu}^a}\big)
	\ln\frac{k_f^a\bm{\omega}^{\bm{\mu}^a}}{k_b^a\bm{\omega}^{\bm{\nu}^a}} dz \\
  &\ge C\int_\Omega\big(k_f^a\bm{\omega}^{\bm{\mu}^a}
	- k_b^a\bm{\omega}^{\bm{\nu}^a}\big)
	\ln\frac{k_f^a\bm{\omega}^{\bm{\mu}^a}}{k_b^a\bm{\omega}^{\bm{\nu}^a}} dz,
\end{align*}
where $C>0$ depends on $M_{\rm max}$.
In the case $\gamma^a < 0$, we argue in the same
way, leading to
\begin{align*}
  \sum_{a=1}^N\int_\Omega(k_f^a\bm{x}^{\bm{\alpha}^a}-k_b^a\bm{x}^{\bm{\beta}^a})
	\ln\frac{k_f^a \bm{x}^{\bm{\alpha}^a}}{k_b^a \bm{x}^{\bm{\beta}^a}}dz
	&= \sum_{a=1}^N\int_\Omega\frac{1}{c^{\sum_{i=1}^n\beta_i^a}}
	\big(k_f^a\bm{\omega}^{\bm{\mu}^a}
	- k_b^a\bm{\omega}^{\bm{\nu}^a}\big)
	\ln\frac{k_f^a\bm{\omega}^{\bm{\mu}^a}}{k_b^a\bm{\omega}^{\bm{\nu}^a}} dz \\
  &\ge C\int_\Omega\big(k_f^a\bm{\omega}^{\bm{\mu}^a} - k_b^a\bm{\omega}^{\bm{\nu}^a}\big)
	\ln\frac{k_f^a\bm{\omega}^{\bm{\mu}^a}}{k_b^a\bm{\omega}^{\bm{\nu}^a}} dz.
\end{align*}
Consequently, taking into account Lemma \ref{lem.D.ge}, we find that
\begin{equation}\label{3.D}
  D[\bm{x}] \ge \widetilde{D}[\bm{\omega}]
	:= C\sum_{i=1}^{n+1}\int_\Omega|\na\omega_i^{1/2}|^2dz
	+ C\sum_{a=1}^N\int_\Omega\big(k_f^a\bm{\omega}^{\bm{\mu}^a}
	-k_b^a\bm{\omega}^{\bm{\nu}^a}\big)
	\ln\frac{k_f^a \bm{\omega}^{\bm{\mu}^a}}{k_b^a \bm{\omega}^{\bm{\nu}^a}}dz.
\end{equation}

We need to determine the conservation laws for $\overline{\bm{\omega}}$.
We write $\bm{1}=(1,\ldots,1)^\top\in\R^{n+1}$.

\begin{lemma}\label{lem.om.cl}
Assume that $\Q\overline{\bm{c}}=\MM$. Then
$\overline{\bm{\omega}}=(\overline{c_1},\ldots,\overline{c_n},\overline{c})$ 
satisfies the conservation laws
$$
  \widehat{\Q}\overline{\bm{\omega}} = \wMM,
$$
where $\widehat{\Q}$ and $\wMM$ are defined by
\begin{equation}\label{3.hat}
  \widehat{\Q} = \begin{pmatrix}
	\Q & \bm{0} \\
	\bm{1}^\top & -1 
	\end{pmatrix}\in\R^{(m+1)\times(n+1)}, \quad
	\wMM = \begin{pmatrix}
	\MM \\ 0 \end{pmatrix}\in\R^{n+1}.
\end{equation}
\end{lemma}

\begin{proof}
The result follows from a direct computation:
$$
  \widehat{\Q}\overline{\bm{\omega}} 
	= \begin{pmatrix}	\Q & \bm{0} \\ \bm{1} & -1 \end{pmatrix}
	\begin{pmatrix} \overline{\omega_1} \\ \vdots \\ \overline{\omega_n} \\
	\overline{\omega_{n+1}} \end{pmatrix}
	= \begin{pmatrix} \Q\overline{\bm{c}} \\
	\sum_{i=1}^n \overline{c_i} - \overline{c} \end{pmatrix}
	= \begin{pmatrix} \MM \\ 0 \end{pmatrix},
$$
since it holds that $\overline{c}=\sum_{i=1}^n\overline{c_i}$.
\end{proof}

\begin{lemma}\label{lem.tech}
There exists a constant $C>0$, depending on $\Omega$, $n$, $N$, $k_f^a$,
$k_b^a$ ($a=1,\ldots,N$), and $M_i$ ($i=1,\ldots,n$), such that
$$
  \widetilde D[\bm{\omega}] \ge C\sum_{a=1}^N\Big(
	(k_f^a)^{1/2}\sqrt{\overline{\bm{\omega}}}^{\bm{\mu}^a}
  - (k_b^a)^{1/2}\sqrt{\overline{\bm{\omega}}}^{\bm{\nu}^a}\Big)^2,
$$
for all measurable functions $\bm{\omega}: \Omega \to \R_+^{n+1}$ such that 
$\widetilde D[\bm{\omega}]$ is finite, with $\widetilde{D}[\bm{\omega}]$ defined in \eqref{3.D}.
\end{lemma}

A similar but slightly simpler result for reaction-diffusion systems is
proved in \cite[Lemma 2.7]{FeTa17a}. 
The proof of this lemma is lengthy and therefore shifted to Appendix \ref{sec.tech}. 
We remark that the validity of this lemma applies to all measurable functions with $
\widetilde{D}[\bm{\omega}] < +\infty$.

\begin{lemma}\label{lem.Dreac.ge}
Assume that \eqref{1.eq}--\eqref{1.reac} possesses no boundary equilibria. 
Fix $\MM \in \R_+^m$ such that $\bm{\zeta}\MM = 1$.
Then there exists a nonconstructive constant $C>0$ such that for all $\overline{\bm{
\omega}}\in\R_+^{n+1}$ satisfying
$\widehat{\Q}\overline{\bm{\omega}}=\wMM$, it holds that
\begin{equation}\label{3.Dreac.ge}
  \sum_{a=1}^N\Big((k_f^a)^{1/2}\sqrt{\overline{\bm{\omega}}}^{\bm{\mu}^a}
	- (k_b^a)^{1/2}\sqrt{\overline{\bm{\omega}}}^{\bm{\nu}^a}\Big)^2
	\ge C\sum_{i=1}^{n+1}\big(\overline{\omega_i}^{1/2}-\omega_{i\infty}^{1/2}\big)^2,
\end{equation}
where $\bm{\omega}_\infty$ is constructed in Proposition \ref{pro.aug}.
\end{lemma}

\begin{remark}\rm
We mark that this lemma is proved for {\em any} vector $\overline{\bm{\omega}}\in 
\R_+^{n+1}$ satisfying the conservation laws. It does not use
any analytical properties of solutions to \eqref{1.eq}-\eqref{1.reac}. 
The notation $\overline{\bm{\omega}}$ is a bit abusive, since we later
apply this lemma to the average $\overline{\bm{\omega}}$, where $\bm{\omega}$ is 
constructed from solutions to \eqref{1.eq}-\eqref{1.reac}.
\end{remark}

\begin{remark}\label{rem.construc}\rm
While all the constants before and after this lemma are constructive, this is
not the case for the constant in Lemma \ref{lem.Dreac.ge}, since the lemma is
proved by using a contradiction argument. Still, inequality \eqref{3.Dreac.ge} is
finite-dimensional. Therefore, in the general case, the rate of convergence to
equilibrium to system \eqref{1.eq}-\eqref{1.reac} is constructive up to the
finite-dimensional inequality \eqref{3.Dreac.ge}. We present in Section \ref{sec.exam}
an example for which \eqref{3.Dreac.ge} can be proved with a constructive 
(even explicit) constant,
which consequently leads to a constructive rate of convergence to equilibrium
for \eqref{1.eq}-\eqref{1.reac}.
\qed
\end{remark}

\begin{proof}[Proof of Lemma \ref{lem.Dreac.ge}]
We first show that $\bm{\overline{\omega}}$ is bounded. Indeed, we infer from
$\widehat{\Q}\overline{\bm{\omega}} = \widehat{\bm{M}}^0$ that
$\Q \overline{\bm{\omega}}' = \MM$. Thus, $1 = \bm{\zeta}\MM = \bm{\zeta}\Q 
\overline{\bm{\omega}} = \sum_{i=1}^{n}M_i \overline{\omega_i}$. Hence,
$\overline{\omega_i} \leq 1/M_{\rm min}$ and consequently 
$\overline{\omega}_{n+1} = \sum_{i=1}^{n}\overline{\omega_i} \leq n/M_{\rm min}$.

We will now prove that 
$$
  \lambda := \inf_{\overline{\bm{\omega}}\in\R_+^{n+1}:
	\widehat{\Q}\overline{\bm{\omega}}=\wMM}
	\frac{\sum_{a=1}^N\big((k_f^a)^{1/2}\sqrt{\overline{\bm{\omega}}}^{\bm{\mu}^a}
	- (k_b^a)^{1/2}\sqrt{\overline{\bm{\omega}}}^{\bm{\nu}^a}\big)^2}{
	\sum_{i=1}^{n+1}\big(\overline{\omega_i}^{1/2}-\omega_{i\infty}^{1/2}\big)^2} > 0.
$$
It is obvious that $\lambda\ge 0$. Since the denominator is bounded from above,
$\lambda=0$ can occur only if the nominator approaches zero. In view of
Proposition \ref{pro.aug} and the fact that the system is assumed to 
have no boundary equilibria, the nominator can converge to zero only when 
$\overline{\bm{\omega}}\to\bm{\omega}_\infty$.
Therefore, $\lambda=0$ is only possible if $\delta=0$, where $\delta$ is the 
linearized version of $\lambda$ defined
in Lemma \ref{lem.delta} below. Setting $\eta_i=\overline{\omega_i}-\omega_{i\infty}$,
Lemma \ref{lem.delta} shows that $\delta=0$ if and only if 
$$
  0 = \liminf_{\widehat{\Q}\overline{\bm{\omega}}=\wMM,\,
	\overline{\bm{\omega}}\to\bm{\omega}_\infty}
	\frac{\sum_{a=1}^N k_f^a\bm{\omega}_\infty^{\bm{\mu}^a}\big\{\sum_{i=1}^{n+1}
	(\mu_i^a-\nu_i^a)\eta_i\omega_{i\infty}^{-1}\big\}^2}{
	\sum_{i=1}^{n+1}\eta_i^2\omega_{i\infty}^{-1}}.
$$
Since the nominator and denominator have the same homogeneity, the limit
inferior remains unchanged if $\bm{\eta}=(\eta_1,\ldots,\eta_{n+1})$ 
has unit length, $\|\bm{\eta}\|_{\R^{n+1}}=1$ (using the Euclidean norm).
We infer from $\widehat{\Q}\overline{\bm{\omega}}=\wMM=\widehat{\Q}
\bm{\omega}_\infty$ that $\widehat{\Q}\bm{\eta}=0$.
Hence, we have $\delta=0$ if and only if there exists a vector $\bm{\eta}\in\R^{n+1}$
satisfying $\|\bm{\eta}\|_{\R^{n+1}}=1$, $\widehat{\Q}\bm{\eta}=0$, and
$$
  \sum_{i=1}^{n+1}(\mu_i^a-\nu_i^a)\frac{\eta_i}{\omega_{i\infty}} = 0
	\quad\mbox{for all }a=1,\ldots,N.
$$
The last identity implies that the vector $\bm{\eta}/\bm{\omega}_\infty:=
(\eta_1/\omega_{1\infty},\ldots,\eta_{n+1}/\omega_{n+1,\infty})^\top$ belongs to
the kernel of $\mathbb{P}^\top$, where 
$$
  \mathbb{P} = \big(\bm{\nu}^a-\bm{\mu}^a\big)_{a=1,\ldots,N}
	\in\R^{(n+1)\times N}.
$$
Since the rows of $\Q$ form a basis of the Wegscheider matrix
$\mathbb{W}=(\bm{\beta}^a-\bm{\alpha}^a)_{a=1,\ldots,N}$, and
taking into account definition \eqref{3.munu} of $\bm{\mu}^a$ and $\bm{\nu}^a$,
we see that the columns of the matrix
$$
  \Q^*:=\begin{pmatrix}
	\Q^\top & \mathbf{1}_n \\
	\mathbf{0} & 1
	\end{pmatrix}
$$
form a basis of $\operatorname{ker}(\mathbb{P}^\top)$.
We deduce that there exists $\bm{\rho}\in\R^{n+1}$ such that
$\bm{\eta}/\bm{\omega}_\infty=\Q^*\bm{\rho}$ or, equivalently,
$\bm{\eta} = \DD\Q^*\bm{\rho}$, where $\DD=\operatorname{diag}(\omega_{1\infty},
\ldots,\omega_{n+1,\infty})$.
Hence, because of $\widehat\Q\bm{\eta}=0$, we obtain $\widehat\Q\DD\Q^*\bm{\rho}=0$.
The idea is now to prove that $\bm{\rho}=0$, which implies that  
$\bm{\eta} = \DD\Q^*\bm{\rho}=0$, contradicting $\|\bm{\eta}\|_{\R^{n+1}}=1$.

We claim that the matrix $\widehat\Q\DD\Q^*$ is invertible. Indeed, 
setting $\mathbb{A}_\infty =\operatorname{diag}(\omega_{1\infty},\ldots,
\omega_{n\infty})$, we compute
$$
  \widehat\Q\DD\Q^*
	= \begin{pmatrix} \Q & \mathbf{0} \\ \mathbf{1}^\top & -1 \end{pmatrix}
	\begin{pmatrix} \mathbb{A}_\infty  & \mathbf{0} \\ 
	\mathbf{0} & \omega_{n+1,\infty} \end{pmatrix}
	\begin{pmatrix} \Q^\top & \mathbf{1} \\ \mathbf{0} & 1 \end{pmatrix}
	= \begin{pmatrix}
	\Q\mathbb{A}_\infty \Q^\top & \Q\mathbb{A}_\infty \mathbf{1} \\
	\mathbf{1}^\top\mathbb{A}_\infty \Q^\top 
	& \mathbf{1}^\top\mathbb{A}_\infty \mathbf{1} - \omega_{n+1,\infty}
	\end{pmatrix}.
$$
Since $\mathbf{1}^\top\mathbb{A}_\infty\mathbf{1}=\sum_{i=1}^{n}\omega_{i\infty}
=\omega_{n+1,\infty}$ (see Proposition \ref{pro.aug}), it follows that
$$
  \widehat\Q\DD\Q^* = \begin{pmatrix}
	\Q\mathbb{A}_\infty \Q^\top & \Q\mathbb{A}_\infty \mathbf{1} \\
	\mathbf{1}^\top\mathbb{A}_\infty \Q^\top & 0
	\end{pmatrix}.
$$
We claim that the matrix $\Q\mathbb{A}_\infty \Q^\top$ is regular. 
Since $\Q$ has full rank, so is $\Q^\top$, and we infer for all $\bm{\xi}\in\R^m$ that
$$
  \big\langle\bm{\xi},\Q\mathbb{A}_\infty \Q^\top\bm{\xi}\big\rangle
	= \big\langle\bm{\xi},
	\Q\mathbb{A}_\infty ^{1/2}\mathbb{A}_\infty ^{1/2}\Q^\top\bm{\xi}\big\rangle
	= \big\langle\mathbb{A}_\infty ^{1/2}\Q^\top\bm{\xi},
	\mathbb{A}_\infty^{1/2}\Q^\top\bm{\xi}\big\rangle \ge 0
$$
with equality if and only if $\bm{\xi}=\bm{0}$. Hence, $\Q\mathbb{A}_\infty \Q^\top$ is 
regular. Together with the rule on the determinant of block matrices, this shows that
$$
  \det(\widehat\Q\DD\Q^*)
	= \det(\Q\mathbb{A}_\infty \Q^\top)\det\big[0-(\mathbf{1}^\top\mathbb{A}_\infty \Q^\top)
	(\Q\mathbb{A}_\infty \Q^\top)^{-1}(\Q\mathbb{A}_\infty \mathbf{1})\big].
$$
As we already know that $\det(\Q\mathbb{A}_\infty \Q^\top)\neq 0$, it remains to verify
that the second factor does not vanish. As the expression in the brackets $[\cdots]$
is a number, we need to show that
\begin{equation}\label{3.aux6}
  (\mathbf{1}^\top\mathbb{A}_\infty \Q^\top)
	(\Q\mathbb{A}_\infty \Q^\top)^{-1}(\Q\mathbb{A}_\infty \mathbf{1}) \neq 0.
\end{equation}
The diagonal matrix $\mathbb{A}_\infty \in\R^{n\times n}$ has strictly positive diagonal
elements. Therefore, \eqref{3.aux6} is equivalent to
$$
  (\mathbf{1}^\top\mathbb{A}_\infty ^{1/2})(\mathbb{A}_\infty ^{1/2}\Q^\top)
	\big((\Q\mathbb{A}_\infty ^{1/2})(\mathbb{A}_\infty ^{1/2}\Q^\top)\big)^{-1}
	(\Q\mathbb{A}_\infty ^{1/2})(\mathbf{1}^\top\mathbb{A}_\infty ^{1/2})^\top \neq 0.
$$
We abbreviate the left-hand side by introducing 
$\bm{z}=\mathbf{1}^\top\mathbb{A}_\infty^{1/2}\in\R^{1\times n}$
and $\mathbb{X}=\mathbb{A}_\infty ^{1/2}\Q^\top\in\R^{n\times m}$. Then \eqref{3.aux6}
becomes
$$
  \bm{z}\mathbb{X}(\mathbb{X}^\top\mathbb{X})^{-1}\mathbb{X}^\top\bm{z}^\top \neq 0.
$$
Since $\mathbb{X}$ is not a square matrix, we cannot invert it, but we may
consider its Moore-Penrose generalized inverse $\mathbb{X}^\dag$; see \cite{Pen55} 
or \cite[Section 11.5]{Ser10} for a definition and properties. 
We compute
\begin{align*}
  \bm{z}\mathbb{X}(\mathbb{X}^\top\mathbb{X})^{-1}\mathbb{X}^\top\bm{z}^\top
	&= \bm{z}\mathbb{X}(\mathbb{X}^\top\mathbb{X})^\dag\mathbb{X}^\top\bm{z}^\top
	\qquad\mbox{\cite[page 218]{Ser10}} \\
	&= \bm{z}\mathbb{X}\mathbb{X}^\dag(\mathbb{X}^\top)^\dag\mathbb{X}^\top\bm{z}^\top
	\qquad\mbox{\cite[Lemma 1.5]{Pen55}} \\
	&= \bm{z}\mathbb{X}\mathbb{X}^\dag(\mathbb{X}^\dag)^\top\mathbb{X}^\top\bm{z}^\top
	\qquad\mbox{\cite[Prop.~11.5]{Ser10}} \\
	&= \bm{z}(\mathbb{X}\mathbb{X}^\dag)(\mathbb{X}\mathbb{X}^\dag)^\top\bm{z}^\top
	\qquad\mbox{\cite[Lemma 1.5]{Pen55}} \\
  &= \|(\mathbb{X}\mathbb{X}^\dag)^\top\bm{z}^\top\|_{\R^n}^2.
\end{align*}
Consequently, \eqref{3.aux6} holds if and only if 
$(\mathbb{X}\mathbb{X}^\dag)^\top\bm{z}^\top\neq 0$ or
$\bm{z}^\top\not\in\operatorname{ker}((\mathbb{X}\mathbb{X}^\dag)^\top)$. 
Now, it holds that
$$
  \operatorname{ker}\big((\mathbb{X}\mathbb{X}^\dag)^\top\big)
	= \operatorname{ker}\big((\mathbb{X}^\dag)^\top\mathbb{X}^\top\big)
	= \operatorname{ker}\big((\mathbb{X}^\top)^\dag\mathbb{X}^\top\big)
	= \operatorname{ker}(\mathbb{X}^\top),
$$
where the last step follows from \cite[page 219]{Ser10}. We infer that
$\bm{z}^\top\not\in\operatorname{ker}((\mathbb{X}\mathbb{X}^\dag)^\top)$ if and only if
$\mathbb{A}_\infty ^{1/2}\mathbf{1}=\bm{z}^\top\not\in\operatorname{ker}(\mathbb{X}^\top)
=\operatorname{ker}(\Q\mathbb{A}_\infty ^{1/2})$, which is equivalent to
$$
  0 \neq(\Q\mathbb{A}_\infty ^{1/2})(\mathbb{A}_\infty ^{1/2}\mathbf{1}) 
	= \Q\mathbb{A}_\infty \mathbf{1}
	= \Q\bm{\omega}_\infty',
$$
and this property holds true since $\Q\bm{\omega}_\infty'=\MM\neq 0$. This proves that
\eqref{3.aux6} holds. As mentioned before, this implies that $\bm{\rho}=0$ 
and consequently $\bm{\eta}=0$, which
contradicts the fact that $\bm{\eta}$ has unit length. We conclude that 
$\delta>0$ (defined in Lemma \ref{lem.delta}) and $\lambda>0$, finishing the proof.
\end{proof}

We now provide the technical computations needed in Lemma \ref{lem.Dreac.ge}.

\begin{lemma}\label{lem.delta}
Let $\bm{\omega}_\infty$ be a positive detailed-balanced equilibrium constructed 
in Proposition \ref{pro.aug}. It holds that
\begin{align*}
  \delta &:= \liminf_{\widehat{\Q}\overline{\bm{\omega}}=\wMM,\,
	\overline{\bm{\omega}}\to\bm{\omega}_\infty}
	\frac{\sum_{a=1}^N\big\{(k_f^a)^{1/2}\sqrt{\overline{\bm{\omega}}}^{\bm{\mu}^a}
	- (k_b^a)^{1/2}\sqrt{\overline{\bm{\omega}}}^{\bm{\nu}^a}\big\}^2}{
	\sum_{i=1}^{n+1}\big(\overline{\omega_i}^{1/2}-\omega_{i\infty}^{1/2}\big)^2} \\
	&= \frac12\liminf_{\widehat{\Q}\overline{\bm{\omega}}=\wMM,\,
	\overline{\bm{\omega}}\to\bm{\omega}_\infty}
	\frac{\sum_{a=1}^N k_f^a\bm{\omega}_\infty^{\bm{\mu}^a}\big\{\sum_{i=1}^{n+1}
	(\mu_i^a-\nu_i^a)(\overline{\omega_i}-\omega_{i\infty})\omega_{i\infty}^{-1}\big\}^2}{
	\sum_{i=1}^{n+1}(\overline{\omega_i}-\omega_{i\infty})^2\omega_{i\infty}^{-1}}.
\end{align*}
\end{lemma}

\begin{proof}
We denote by
\begin{align*}
  D_1(\overline{\bm{\omega}})
	&= \sum_{a=1}^N\Big((k_f^a)^{1/2}\sqrt{\overline{\bm{\omega}}}^{\bm{\mu}^a}
	- (k_b^a)^{1/2}\sqrt{\overline{\bm{\omega}}}^{\bm{\nu}^a}\Big)^2, \\
	D_2(\overline{\bm{\omega}})
	&= \sum_{i=1}^{n+1}\big(\overline{\omega_i}^{1/2}-\omega_{i\infty}^{1/2}
	\big)^2
\end{align*}
the nominator and denominator of the definition of $\delta$, respectively. 
We linearize both expressions around $\bm{\omega}_\infty$,
\begin{equation}\label{3.Di}
\begin{aligned}
  D_i(\overline{\bm{\omega}}) &= D_i(\bm{\omega}_\infty)
	+ \nabla D_i(\bm{\omega}_\infty)
	\cdot(\overline{\bm{\omega}}-\bm{\omega}_\infty) \\
	&\phantom{xx}{}+ \frac12(\overline{\bm{\omega}}-\bm{\omega}_\infty)^\top
	\nabla^2D_i(\bm{\omega}_\infty)(\overline{\bm{\omega}}-\bm{\omega}_\infty) 
	+ o(|\overline{\bm{\omega}}-\bm{\omega}_\infty|^2).
\end{aligned}
\end{equation}
Since $\bm{\omega}_\infty$ is a detailed-balanced equilibrium, it holds that
$(k_f^a)^{1/2}\sqrt{\bm{\omega}_\infty}^{\bm{\mu}^a}
= (k_b^a)^{1/2}\sqrt{\bm{\omega}_\infty}^{\bm{\nu}^a}$ for
all $a=1,\ldots,N$, implying that $D_1(\bm{\omega}_\infty)=0$ and
$\nabla D_1(\bm{\omega}_\infty)=0$. Let $\pa_i=\pa/\pa\omega_i$. Then
\begin{align*}
  &\pa_j\pa_i D_1(\overline{\bm{\omega}}) \\
	&= \sum_{a=1}^N\bigg\{\pa_j\pa_i
	\Big((k_f^a)^{1/2}\sqrt{\overline{\bm{\omega}}}^{\bm{\mu}^a}
	- (k_b^a)^{1/2}\sqrt{\overline{\bm{\omega}}}^{\bm{\nu}^a}\Big) 
	\Big((k_f^a)^{1/2}\sqrt{\overline{\bm{\omega}}}^{\bm{\mu}^a}
	- (k_b^a)^{1/2}\sqrt{\overline{\bm{\omega}}}^{\bm{\nu}^a}\Big) \\
  &\phantom{xx}{}+ \pa_i
	\Big((k_f^a)^{1/2}\sqrt{\overline{\bm{\omega}}}^{\bm{\mu}^a}
	- (k_b^a)^{1/2}\sqrt{\overline{\bm{\omega}}}^{\bm{\nu}^a}\Big) 
	\pa_j\Big((k_f^a)^{1/2}\sqrt{\overline{\bm{\omega}}}^{\bm{\mu}^a}
	- (k_b^a)^{1/2}\sqrt{\overline{\bm{\omega}}}^{\bm{\nu}^a}\Big)\bigg\}.
\end{align*}
The first term vanishes for $\overline{\bm{\omega}}=\bm{\omega}_\infty$, and
for the second term we compute
\begin{align*}
  \pa_i	\Big( & (k_f^a)^{1/2}\sqrt{\overline{\bm{\omega}}}^{\bm{\mu}^a}
	- (k_b^a)^{1/2}\sqrt{\overline{\bm{\omega}}}^{\bm{\nu}^a}\Big) \\
	&= (k_f^a)^{1/2}\pa_i\prod_{k=1}^{n+1}\overline{\omega_k}^{\mu_k^a/2}
	- (k_b^a)^{1/2}\pa_i\prod_{k=1}^{n+1}\overline{\omega_k}^{\nu_k^a/2} \\
  &= (k_f^a)^{1/2}\frac{\mu_i^a}{2}\frac{1}{\overline{\omega_i}}
	\prod_{k=1}^{n+1}\overline{\omega_k}^{\mu_k^a/2}
	- (k_b^a)^{1/2}\frac{\nu_i^a}{2}\frac{1}{\overline{\omega_i}}
	\prod_{k=1}^{n+1}\overline{\omega_k}^{\nu_k^a/2} \\
	&= \frac{1}{2\overline{\omega_i}}\Big((k_f^a)^{1/2}\mu_i^a
	\sqrt{\overline{\bm{\omega}}}^{\bm{\mu}^a} 
	- (k_b^a)^{1/2}\nu_i^a\sqrt{\overline{\bm{\omega}}}^{\bm{\nu}^a}\Big).
\end{align*}
Evaluating this expression at $\overline{\bm{\omega}}=\bm{\omega}_\infty$ and
using $(k_f^a)^{1/2}\sqrt{\bm{\omega}_\infty}^{\bm{\mu}^a}=(k_b^a)^{1/2}
\sqrt{\bm{\omega}_\infty}^{\bm{\nu}^a}$, it follows that
$$
  \pa_i	\Big((k_f^a)^{1/2}\sqrt{\overline{\bm{\omega}}}^{\bm{\mu}^a}
	- (k_b^a)^{1/2}\sqrt{\overline{\bm{\omega}}}^{\bm{\nu}^a}\Big)
	\Big|_{\overline{\bm{\omega}}=\bm{\omega}_\infty}
	= \frac12\frac{\mu_i^a-\nu_i^a}{\omega_{i\infty}}
	(k_f^a)^{1/2}\sqrt{\bm{\omega}_\infty}^{\bm{\mu}^a}.
$$
Consequently,
$$
  \pa_j\pa_i D_1(\bm{\omega}_\infty)
	= \frac14\sum_{a=1}^N k_f^a\bm{\omega}_\infty^{\bm{\mu}^a}
	\frac{\mu_i^a-\nu_i^a}{\omega_{i\infty}}
	\frac{\mu_j^a-\nu_j^a}{\omega_{j\infty}},
$$
and the quadratic term in the Taylor expansion becomes at the point
$\bm{\omega}_\infty$
$$
 \frac12(\overline{\bm{\omega}}-\bm{\omega}_\infty)^\top
	\nabla^2D_i(\bm{\omega}_\infty)(\overline{\bm{\omega}}-\bm{\omega}_\infty)
	= \frac18\sum_{a=1}^N k_f^a\bm{\omega}_\infty^{\bm{\mu}^a}\bigg(\sum_{i=1}^{n+1}
	\frac{\mu_i^a-\nu_i^a}{\omega_{i\infty}}\big(\overline{\omega_i}-\omega_{i\infty}\big)
	\bigg)^2.
$$
Similarly, $D_2(\bm{\omega}_\infty)=0$, $\nabla D_2(\bm{\omega}_\infty)=0$,
and
$$
  \frac12(\overline{\bm{\omega}}-\bm{\omega}_\infty)^\top
	\nabla^2D_2(\bm{\omega}_\infty)(\overline{\bm{\omega}}-\bm{\omega}_\infty)
	= \frac14\sum_{i=1}^{n+1}\frac{(\overline{\omega_i}-\omega_{i\infty})^2}{
	\omega_{i\infty}}.
$$
We insert these expressions into \eqref{3.Di} and compute
$D_1(\overline{\bm{\omega}})/D_2(\overline{\bm{\omega}})$. The limit
$\overline{\bm{\omega}}\to\bm{\omega}_\infty$ such that
$\widehat{\Q}\overline{\bm{\omega}}=\wMM$ then gives the conclusion.
\end{proof}

We are ready to prove the main result of this subsection.

\begin{proposition}[Entropy entropy-production inequality; unequal homogeneities]\label{prop.uneqhom}\mbox{ }\\
Fix $\MM \in \R_+^m$ such that $\bm{\zeta}\MM = 1$. Let $\bm{x}_\infty$ be 
the equilibrium constructed in Theorem \ref{thm.equi}.
Assume that \eqref{3.uneqhom} holds and system \eqref{1.eq}--\eqref{1.reac} has no 
boundary equilibria. Then there exists a constant $\lambda>0$,
which is constructive up to a finite-dimensional inequality (in the sense of
Remark \ref{rem.construc}), such that
$$
  D[\bm{x}]\ge\lambda E[\bm{x}|\bm{x}_\infty]
$$
for all functions $\bm{x}: \Omega \to \R_+^n$ having the same regularity as 
the corresponding solutions
in Theorem \ref{thm.ex} and satisfying $\Q\overline{\bm{c}} = \MM$.
\end{proposition}

\begin{proof}
Lemma \ref{lem.E.le} shows that
\begin{equation}\label{3.aux7}
  E[\bm{x}|\bm{x}_\infty] \le C\sum_{i=1}^n\bigg(\int_\Omega
	\Big(c_i^{1/2}-\overline{c_i^{1/2}}\,\Big)^2 dz
	+ \big(\overline{c_i}^{1/2}-c_{i\infty}^{1/2}\big)^2\bigg).
\end{equation}
The first sum is controlled by $D[\bm{x}]$ using Lemma \ref{lem.D.ge} and
the Poincar\'e inequality (with constant $C_P>0$):
$$
  D[\bm{x}] \ge \sum_{i=1}^n\int_\Omega|\na c_i^{1/2}|^2 dz
	\ge C_p\sum_{i=1}^n\int_\Omega\big(c_i^{1/2}-\overline{c_i^{1/2}}\big)^2 dz.
$$
The second sum on the right-hand side is estimated by combining 
estimate \eqref{3.D}, Lemma \ref{lem.tech}, and Lemma \ref{lem.Dreac.ge}:
$$
  D[\bm{x}] \ge C\sum_{i=1}^{n+1}\big(\overline{\omega_i}^{1/2}
	-\omega_{i\infty}^{1/2}\big)^2 \ge C\sum_{i=1}^n
	\big(\overline{c_i}^{1/2}-c_{i\infty}^{1/2}\big)^2.
$$
Adding the previous two inequalities and using \eqref{3.aux7} then
concludes the proof.
\end{proof}


\subsection{Proof of Theorem \ref{thm.main}}\label{sec.proof}

The starting point is the discrete entropy inequality (see Remark \ref{rem.dei}):
$$
  E[\bm{x}^k|\bm{x}_\infty] + \tau D[\bm{x}^k]
	+ C\varepsilon\tau\sum_{i=1}^{n-1}\|w_i^k\|_{H^l(\Omega)}^2
	\le E[\bm{x}^{k-1}|\bm{x}_\infty].
$$
Using the entropy-production inequality from Propositions \ref{prop.eqhom} 
or \ref{prop.uneqhom}, this becomes
$$
  E[\bm{x}^k|\bm{x}_\infty] \le (1+\lambda\tau)^{-1}E[\bm{x}^{k-1}|\bm{x}_\infty]
$$
and, by induction,
$$
  E[\bm{x}^k|\bm{x}_\infty] \le (1+\lambda\tau)^{-k}E[\bm{x}^0|\bm{x}_\infty]
	= (1+\lambda\tau)^{-T/\tau}E[\bm{x}^0|\bm{x}_\infty].
$$
Performing the limit $\tau\to 0$ or, equivalently, 
$k\to\infty$, we find that
$$
  E[\bm{x}(T)|\bm{x}_\infty]\le\liminf_{k\to\infty}E[\bm{x}^k|\bm{x}_\infty]
	\le e^{-\lambda T}E[\bm{x}^0|\bm{x}_\infty].
$$
Clearly, this inequality also holds for $t\in(0,T)$ instead of $T$. Then,
by the Csisz\'ar--Kullback--Pinsker inequality in Lemma \ref{lem.CKP}, with
constant $C_{\rm CKP}>0$,
$$
  \sum_{i=1}^n\|x_i(t)-x_{i\infty}\|_{L^1(\Omega)}^2
	\le \frac{e^{-\lambda t}}{C_{\rm CKP}}\int_\Omega h(\bm{\rho}'(0))dz.
$$
As $x_i$ is bounded in $L^\infty(0,\infty;L^\infty(\Omega))$, we derive the
convergence in $L^p$ for $1\le p<\infty$ from an interpolation argument
\begin{align*}
  \sum_{i=1}^n\|x_i(t)-x_{i\infty}\|_{L^p(\Omega)}
	&\le \sum_{i=1}^n\|x_i(t)-x_{i\infty}\|_{L^\infty(\Omega)}^{1-1/p}
	\|x_i(t)-x_{i\infty}\|_{L^1(\Omega)}^{1/p} \\
	&\le Ce^{-\lambda t/(2p)}, \quad t>0,
\end{align*}
which concludes the proof.


\section{Example: a specific reaction}\label{sec.exam}

As mentioned in Remark \ref{rem.construc}, the rate of convergence to equilibrium
is generally not constructive since the finite-dimensional inequality
\eqref{3.Dreac.ge} is proved by a nonconstructive contradiction argument.
The derivation of a constructive constant for this inequality seems to be a
challenging problem, which goes beyond the scope of this paper. In this section,
we show that, potentially in any specific system, the finite-dimensional
inequality \eqref{3.Dreac.ge} can be proved in a constructive way and thus gives
the exponential decay with constructive constant. More specifically, we consider
the single reversible reaction
$$
	A_1 + A_2 \leftrightharpoons A_3.
$$
We assume for simplicity that the forward and backward reaction constants equal one.
Furthermore, $|\Omega|=1$. The corresponding system reads as
\begin{align}
		\partial_t \rho_1 + \diver\bm{j}_1 &= r_1(\bm{x}) = -M_1(x_1x_2 - x_3),
		\nonumber \\
		\partial_t \rho_2 + \diver\bm{j}_2 &= r_2(\bm{x}) = -M_2(x_1x_2 - x_3),
		\label{3x3} \\
		\partial_t \rho_3 + \diver\bm{j}_3 &= r_3(\bm{x}) = +M_3(x_1x_2 - x_3),
		\nonumber
\end{align}
We conclude from total mass conservation $r_1+r_2+r_3=0$, that
$M_1+M_2=M_3$. There are two (formal) conservation laws. The first one follows from
$$
  \frac{d}{dt}\int_\Omega\big(c_1(t)+c_3(t)\big)dz 
	= \frac{d}{dt}\int_\Omega\bigg(\frac{\rho_1(t)}{M_1}+\frac{\rho_3(t)}{M_3}\bigg)dz = 0,
$$
leading to
$$
  \overline{c_1}(t) + \overline{c_3}(t) = M_{13} := \overline{c_1^0} + \overline{c_3^0},
$$
where $\overline{c_i^0}=\overline{\rho_i^0}/M_i = \int_\Omega\rho_i^0dz/M_i$.
The second conservation law reads as
$$
  \overline{c_2}(t) + \overline{c_3}(t) = M_{23} := \overline{c_2^0} + \overline{c_3^0}.
$$
The matrix $\Q$ in this case is
$$
  \Q = \begin{pmatrix} 1& 0 & 1 \\ 0 & 1 & 1 \end{pmatrix},
$$
and we can choose $\bm{\zeta}=(M_1,M_2)$ since the conservation of total mass,
$M_1+M_2=M_3$, gives $\bm{\zeta}\Q=(M_1,M_2,M_3)=\bm{M}^\top$.
The initial mass vector $\MM=(M_{13},M_{23})^\top$ satisfies
$\bm{\zeta}\MM=M_1M_{13}+M_2M_{23}=1$.
It is not difficult to check that the system is detailed balanced and possesses
no boundary equilibria, and thus, for any fixed masses $M_{13}>0$, $M_{23}>0$,
there exists a unique positive detailed-balanced equilibrium  
$\bm{x}_\infty=(x_{1\infty},x_{2\infty},x_{3\infty})^\top\in(0,1)^3$ satisfying
\begin{equation}\label{4.xinfty}
\begin{aligned}
  & x_{1\infty}x_{2\infty} = x_{3\infty}, \quad 
	x_{1\infty} + x_{2\infty} + x_{3\infty} = 1, \\
	& c_{1\infty}+c_{3\infty} = M_{13}, \quad c_{2\infty}+c_{3\infty} = M_{23},
\end{aligned}
\end{equation}
where $c_{i\infty}=c_\infty x_{i\infty}$ and $c_\infty=(M_1x_{1\infty}
+ M_2x_{2\infty} + M_3x_{3\infty})^{-1}$.
We claim that we can prove Lemma \ref{lem.Dreac.ge} with a constructive constant.
More precisely, we show the following result.

\begin{lemma}
There exists a {\rm constructive constant} $C_0>0$, 
only depending on $c_{i\infty}$ and the upper bounds
of $\overline{c_i}$ ($i=1,2,3$), such that
\begin{equation}\label{4.Dreac.ge}
  \big(\sqrt{\overline{c_1}}\sqrt{\overline{c_2}}
 	- \sqrt{\overline{c_3}}\sqrt{\overline{c}}\big)^2
	\ge C_0\sum_{i=1}^3\big(\sqrt{\overline{c_i}}-\sqrt{c_{i\infty}}\big)^2
\end{equation}
for all nonnegative numbers $\overline{c_i}$ and $\overline{c}$ satisfying
\begin{align}
  & \overline{c_1} + \overline{c_3} = M_{13} = c_{1\infty}+c_{3\infty}, \nonumber \\
	& \overline{c_2} + \overline{c_3} = M_{23} = c_{2\infty}+c_{3\infty}, \label{4.cl} \\
  & \overline{c_1} + \overline{c_2} + \overline{c_3} = \overline{c}. \nonumber
\end{align}
\end{lemma}

\begin{proof}
We introduce new variables $\mu_1,\mu_2,\mu_3,\eta\in[-1,\infty)$ by
$$
  \overline{c_i} = c_{i\infty}(1+\mu_i)^2 \quad\mbox{for }i=1,2,3, \quad
	\overline{c} = c_\infty(1+\eta)^2,
$$
recalling that $c_\infty=c_{1\infty}+c_{2\infty}+c_{3\infty}$. 
The uniform bounds for $\overline{c_i}$ show that there exists a constant 
$\mu_{\rm max}>0$ such that $|\mu_i|\le\mu_{\rm max}$ for $i=1,2,3$.
Then the left-hand side of \eqref{4.Dreac.ge} can be formulated as
\begin{align*}
  \big(\sqrt{\overline{c_1}}\sqrt{\overline{c_2}}
 	- \sqrt{\overline{c_3}}\sqrt{\overline{c}}\big)^2
	&= \Big(c_{1\infty}^{1/2}c_{2\infty}^{1/2}(1+\mu_1)(1+\mu_2) 
	- c_{3\infty}^{1/2}c_\infty^{1/2}(1+\mu_3)(1+\eta)\big)^2 \\
	&= c_{1,\infty}c_{2\infty}\big((1+\mu_1)(1+\mu_2)-(1+\mu_3)(1+\eta)\big)^2,
\end{align*}
where we have used $c_{1\infty}c_{2\infty}=x_{1\infty}x_{2\infty}c_\infty^2
= x_{3\infty}c_\infty^2 = c_{3\infty}c_\infty$, which follows from 
$x_{i\infty}=c_{i\infty}/c_\infty$ and the first equation in \eqref{4.xinfty}. 
Furthermore, the right-hand side of \eqref{4.Dreac.ge} is estimated from above by
$$
  \sum_{i=1}^3\big(\sqrt{\overline{c_i}}-\sqrt{c_{i\infty}}\big)^2
	= \sum_{i=1}^3 c_{i\infty}\mu_i^2 \le \max_{i=1,2,3}c_{i\infty}\sum_{i=1}^3 \mu_i^2.
$$
Therefore, it remains to prove the inequality
\begin{equation}\label{4.ineq}
  \big((1+\mu_1)(1+\mu_2)-(1+\mu_3)(1+\eta)\big)^2 \ge C^*\sum_{i=1}^3 \mu_i^2
\end{equation}
for some constructive constant $C^*>0$. 

In terms of the new variables $\mu_i$, the conservation laws in \eqref{4.cl} 
can be written as
\begin{equation}\label{4.id1}
\begin{aligned}
  c_{1\infty}(\mu_1^2+2\mu_1) + c_{3\infty}(\mu_3^2+2\mu_3) &= 0, \\
  c_{2\infty}(\mu_2^2+2\mu_2) + c_{3\infty}(\mu_3^2+2\mu_3) &= 0.
\end{aligned}
\end{equation}
Together with the last equation in \eqref{4.cl}, we obtain
\begin{equation}\label{4.id2}
  c_{1\infty}(\mu_1^2+2\mu_1) = c_{2\infty}(\mu_2^2+2\mu_2) = c_\infty(\eta^2+2\eta).
\end{equation}
Since $\mu_i\ge -1$ and $\eta\ge -1$, we deduce from \eqref{4.id1} and \eqref{4.id2}
that $\mu_1$, $\mu_2$, and $\eta$ always have the same sign and $\mu_3$ has the
opposite sign. We consider therefore two cases:

{\em Case 1: $\mu_1$, $\mu_2$, $\eta\ge 0$ and $\mu_3\le 0$.} Since
$\eta^2+2\eta\ge 0$ and $c_\infty=c_{1\infty}+c_{2\infty}+c_{3\infty}$, it follows 
from \eqref{4.id2} that
$$
  c_{1\infty}(\mu_1^2+2\mu_1) = c_\infty(\eta^2+2\eta) \ge c_{1\infty}(\eta^2+2\eta)
$$
and hence $\mu_1\ge\eta$ (as $z\mapsto z^2+2z$ is increasing on $[-1,\infty)$).
Similarly, we find that $\mu_2\ge\eta$. Therefore,
$$
  (1+\mu_1)(1+\mu_2)-(1+\mu_3)(1+\eta)
	= (\mu_1-\eta) + \mu_2 + \mu_1\mu_2 + (-\mu_3) + (-\mu_3)\eta \ge 0.
$$
Taking the square of this equation, it follows that
\begin{align*}
  \big((1+\mu_1)(1+\mu_2)-(1+\mu_3)(1+\eta)\big)^2
	&\ge \big((\mu_1-\eta)+\mu_2+(-\mu_3)\big)^2 \\
	&\ge (\mu_1-\eta)^2 + \mu_2^2 + (-\mu_3)^2 \ge \mu_2^2+\mu_3^2.
\end{align*}
Exchanging the roles of $\mu_1$ and $\mu_2$, we find that
$$
  \big((1+\mu_1)(1+\mu_2)-(1+\mu_3)(1+\eta)\big)^2 \ge \mu_1^2+\mu_3^2.
$$
Adding these inequalities, we have proved \eqref{4.ineq} with $C^*=\frac12$.

{\em Case 2: $\mu_1$, $\mu_2$, $\eta\le 0$ and $\mu_3\ge 0$.} Because of
$\eta^2+2\eta\le 0$, we have
$$
  c_{1\infty}(\mu_1^2+2\mu_1) = c_\infty(\eta^2+2\eta) \le c_{1\infty}(\eta^2+2\eta),
$$
which yields $\mu_1\le\eta$. Similarly, $\mu_2\le\eta$. A similar argument
as in case 1 leads to
$$
  (1+\mu_3)(1+\eta) - (1+\mu_1)(1+\mu_2)
	= \mu_3(1+\eta) + (\eta-\mu_1) + (-\mu_2)(1+\mu_1) \ge 0.
$$
Hence, taking the square,
\begin{align}
  \big((1+\mu_1)(1+\mu_2)-(1+\mu_3)(1+\eta)\big)^2
	&\ge \big(\mu_3(1+\eta) + (\eta-\mu_1) + (-\mu_2)(1+\mu_1)\big)^2 \nonumber \\
	&\ge \mu_3^2(1+\eta)^2. \label{4.aux}
\end{align}
We deduce from \eqref{4.id2} that
\begin{align*}
  c_\infty(1+\eta)^2 = c_\infty + c_\infty(\eta^2+2\eta) 
	&= c_\infty + c_{1\infty}(\mu_1^2+2\mu_1) \\
	&= c_{2\infty} + c_{3\infty} + c_{1\infty}(1+\mu_1)^2.
\end{align*}
Consequently, $(1+\eta)^2\ge (c_{2\infty}+c_{3\infty})/c_\infty$
and \eqref{4.aux} becomes
\begin{equation}\label{4.aux2}
  \big((1+\mu_1)(1+\mu_2)-(1+\mu_3)(1+\eta)\big)^2 
	\ge \frac{c_{2\infty}+c_{3\infty}}{c_\infty}\mu_3^2.
\end{equation}
We infer from $c_{3\infty}(\mu_3^2+2\mu_3)=-c_{1\infty}(\mu_1^2+2\mu_1)$
(see \eqref{4.id1}) that 
$$
  \mu_3 = \frac{c_{1\infty}(\mu_1+2)}{c_{3\infty}(\mu_3+2)}(-\mu_1)
	\ge \frac{c_{1\infty}}{c_{3\infty}(\mu_{\rm max}+2)}(-\mu_1) \ge 0,
$$
where $\mu_{\rm max}=\max_{i=1,2,3}\mu_i$. Taking the square gives
$$
  \mu_3^2 \ge \frac{c_{1\infty}^2}{c_{3\infty}^2(\mu_{\rm max}+2)^2}\mu_1^2,
$$
and similarly,
$$
  \mu_3^2 \ge \frac{c_{2\infty}^2}{c_{3\infty}^2(\mu_{\rm max}+2)^2}\mu_2^2.
$$
We employ these bounds in \eqref{4.aux2} to obtain
$$
  \big((1+\mu_1)(1+\mu_2)-(1+\mu_3)(1+\eta)\big)^2 
	\ge C^*(\mu_1^2+\mu_2^2+\mu_3^2),
$$
where
$$
  C^* = \frac13\min\bigg\{\frac{c_{2\infty}+c_{3\infty}}{c_\infty},
	\frac{c_{1\infty}^2}{c_{3\infty}^2(\mu_{\rm max}+2)^2},
	\frac{c_{2\infty}^2}{c_{3\infty}^2(\mu_{\rm max}+2)^2}\mu_2^2\bigg\}.
$$
This proves \eqref{4.ineq} and completes the proof.
\end{proof}


\section{Convergence to equilibrium for complex-balanced systems}\label{sec.complex}

One of the main assumptions of this paper is the detailed-balanced condition 
\eqref{1.db}. This condition was used extensively in the thermodynamic community
and it leads to a natural entropy functional that is the core tool for the
global existence analysis and the large-time asymptotics. However, the detailed-balance
condition requires that the reaction system is reversible which is quite restrictive.
In chemical reaction network theory, it is well known that there exists a much
larger class of reaction systems, namely so-called {\em complex-balanced systems}
which also exhibits an entropy structure; see, e.g., \cite{DFT17,FPT17,FeTa17a}
for reaction-diffusion systems. In this section, we show that the global
existence and large-time behavior results can be extended to systems satisfying 
the complex-balanced condition. We only highlight the differences of the
proofs and present full proofs only when necessary.

Consider $n$ constitutents $A_i$ reacting in the following $N$ reactions,
$$
  y_{1,a} A_1 + \cdots + y_{n,a} A_n \xrightarrow{k^a}
	y'_{1,a} A_1 + \cdots + y'_{n,a} A_n \quad\mbox{for }a=1,\ldots,N,
$$
where $k^a>0$ is the reaction rate constant and 
$y_{i,a}$, $y'_{i,a}\in\{0\}\cup[1,\infty)$ are the stoichiometric coefficients.
We set $\bm{y}_a=(y_{1,a},\ldots,y_{n,a})$ and $\bm{y}'_a=(y'_{1,a},\ldots,
y'_{n,a})$. We denote by $\mathcal{C}=\{\bm{y}_a,\bm{y}'_a\}_{a=1,\ldots,N}$
the set of all complexes. We use as in \cite{DFT17}
the convention that the primed complexes
$\bm{y}_a'\in\mathcal{C}$ denote the product of the $a$th reaction, and the
unprimed complexes $\bm{y}_a\in\mathcal{C}$ denote the reactant.
Note that it may happen that $\bm{y}_a=\bm{y}'_b$
for some $a$, $b\in\{1,\ldots,N\}$. This means that a complex can be a reactant
for one reaction and a product for another reaction.

The Maxwell--Stefan diffusion system consists of equations
\eqref{1.eq}, \eqref{1.bic}, and 
\begin{equation}\label{5.reac}
  r_i(\bm{x}) = M_i\sum_{a=1}^N k^a(y'_{i,a}-y_{i,a})\bm{x}^{\bm{y}_a}\quad
	\mbox{with }\bm{x}^{\bm{y}_a} = \prod_{i=1}^n x_i^{y_{i,a}}.
\end{equation}
We assume again the conservation of total mass, expressed as
$$
  \sum_{i=1}^n r_i(\bm{x}) = 0.
$$

\begin{definition}[Complex-balance condition]\label{def.cbc}
A homogeneous equilibrium state $\bm{x}_\infty$ is called a {\em complex-balanced
equilibrium} if for any $\bm{y}\in\mathcal{C}$, it holds that
\begin{equation}\label{5.cbc}
  \sum_{a\in\{1,\ldots,N\}:\bm{y}_a=\bm{y}} k^a\bm{x}_\infty^{\bm{y}_a}
	= \sum_{b\in\{1,\ldots,N\}:\bm{y}'_b=\bm{y}} k^b\bm{x}_\infty^{\bm{y}_b}.
\end{equation}
\end{definition}

Roughly speaking, $\bm{x}_\infty$ is a complex-balanced equilibrium if for any
complex $\bm{y}\in\mathcal{C}$ the total input into each complex balances
the total flow out of the complex. The condition is weaker than detailed balance
since it does not require each step in the forward reaction to be balanced
by a reverse reaction. We say that system \eqref{1.eq}, \eqref{1.bic}, and 
\eqref{5.reac} is a complex-balanced system if it admits a positive complex-balanced
equilibrium. Already Boltzmann studied complex-balanced systems in the context
of kinetic theory, under the name of semi-detailed balance \cite{Bol87}. 
For chemical reaction systems, this condition was systematically studied in 
\cite{FeHo74,Hor72}.

The existence of global weak solutions to \eqref{1.eq}, \eqref{1.bic}, and 
\eqref{5.reac} follows as in Section \ref{sec.ex}. We just have to verify
that Lemma \ref{lem.r} also holds in the case of the reaction terms \eqref{5.reac}.

\begin{lemma}\label{lem.r2}
Let $\bm{x}_\infty$ be a positive complex-balanced equilibrium and 
let the entropy variable $\bm{w}\in\R^{n-1}$ be defined by 
$w_i=\pa h/\pa\rho_i$, $i=1,\ldots,n-1$,
where $h$ is given by \eqref{1.h}. Then for all $\bm{x}\in\R^n$, considered as
a function of $\bm{w}$,
$$
  \sum_{i=1}^{n-1}r_i(\bm{x})w_i \le 0.
$$
\end{lemma}

\begin{proof}
By \eqref{3.rr} and definition \eqref{5.reac} of $r_i$, we compute
\begin{align*}
  \sum_{i=1}^{n-1}r_i(\bm{x})w_i
	&= \sum_{i=1}^n\frac{r_i(\bm{x})}{M_i}\ln\frac{x_i}{x_{i\infty}}
	= \sum_{i=1}^n\sum_{a=1}^N k^a(y_{i,a}'-y_{i,a})\bm{x}^{\bm{y}_a}
	\ln\frac{x_i}{x_{i\infty}} \\
  &= \sum_{a=1}^N k^a\bm{x}^{\bm{y}_a}
	\ln\frac{\bm{x}^{\bm{y}'_a-\bm{y}}}{\bm{x}_\infty^{\bm{y}'_a-\bm{y}}} \\
	&= -\sum_{a=1}^N k^a\bm{x}_\infty^{\bm{y}_a}\bigg\{
	\frac{\bm{x}^{\bm{y}_a}}{\bm{x}_\infty^{\bm{y}_a}}
	\ln\bigg(\frac{\bm{x}^{\bm{y}_a}}{\bm{x}_\infty^{\bm{y}_a}}\biggl/
	\frac{\bm{x}^{\bm{y}'_a}}{\bm{x}_\infty^{\bm{y}'_a}}\bigg)
	- \frac{\bm{x}^{\bm{y}_a}}{\bm{x}_\infty^{\bm{y}_a}}
	+ \frac{\bm{x}^{\bm{y}'_a}}{\bm{x}_\infty^{\bm{y}'_a}}\bigg\} \\
	&\phantom{xx}{}- \sum_{a=1}^N k^a\bm{x}_\infty^{\bm{y}_a}
	\bigg(\frac{\bm{x}^{\bm{y}_a}}{\bm{x}_\infty^{\bm{y}_a}} 
	- \frac{\bm{x}^{\bm{y}'_a}}{\bm{x}_\infty^{\bm{y}'_a}}\bigg).
\end{align*}
The expression in the curly brackets $\{\cdots\}$ 
equals $\Psi(\bm{x}^{\bm{y}_a}/\bm{x}_\infty^{\bm{y}_a},\bm{x}^{\bm{y}'_a}/
\bm{x}_\infty^{\bm{y}'_a})$, where $\Psi(x,y)=x\ln(x/y)-x+y$ is a nonnegative function.
Hence, the first expression on the right-hand side is nonpositive. We claim that
the second expression vanishes. Then $\sum_{i=1}^{n-1}r_i(\bm{x})w_i\le 0$.
Indeed, by the complex-balanced condition 
\eqref{5.cbc},
\begin{align*}
  \sum_{a=1}^N k^a\bm{x}_\infty^{\bm{y}_a}
	\bigg(\frac{\bm{x}^{\bm{y}_a}}{\bm{x}_\infty^{\bm{y}_a}} 
	- \frac{\bm{x}^{\bm{y}'_a}}{\bm{x}_\infty^{\bm{y}'_a}}\bigg)
	&= \sum_{\bm{x}\in\mathcal{C}}\bigg(\sum_{a:\bm{y}_a=\bm{y}}k^a\bm{x}^{\bm{y}_a}
	- \sum_{b:\bm{y}'_b=\bm{y}}k^b\bm{x}_\infty^{\bm{y}_b}
	\frac{\bm{x}^{\bm{y}'_b}}{\bm{x}_\infty^{\bm{y}'_b}}\bigg) \\
  &= \sum_{\bm{y}\in\mathcal{C}}\bigg(\bm{x}^{\bm{y}}\sum_{a:\bm{y}_a=\bm{y}}
	k^a - \frac{\bm{x}^{\bm{y}}}{\bm{x}_\infty^{\bm{y}}}\sum_{b:\bm{y}'_b=\bm{y}}
	k^b\bm{x}_\infty^{\bm{y}_b}\bigg) \\
	&= \sum_{\bm{y}\in\mathcal{C}}\frac{\bm{x}^{\bm{y}}}{\bm{x}_\infty^{\bm{y}}}
	\bigg(\sum_{a:\bm{y}_a=\bm{y}}k^a\bm{x}_\infty^{\bm{y}_a}
	- \sum_{b:\bm{y}_b'=\bm{y}}k^b\bm{x}_\infty^{\bm{y}_b}\bigg) = 0.
\end{align*}
This shows the claim and ends the proof.
\end{proof}

Next, we show the existence of a unique complex-balanced equilibrium. For this,
we denote as before by $\mathbb{W}=(\bm{y}'_a-\bm{y}_a)_{a=1,\ldots,N}\in\R^{n\times N}$
the Wegscheider matrix, set $m=\dim(\operatorname{ker}\mathbb{W})>0$, and denote by 
$\Q\in\R^{m\times n}$ the matrix whose rows form a basis of 
$\operatorname{ker}(\mathbb{W}^\top)$. As in Section \ref{sec.cl}, the
conservation laws are given by
$$
  \Q\overline{\bm{c}}(t) = \MM := \Q\overline{\bm{c}^0}, \quad t>0,
$$
and there exists $\bm{\zeta}\in\R^{1\times m}$ such that $\bm{\zeta}\Q=\bm{M}^\top$
and $\bm{\zeta}\MM=1$.

\begin{proposition}[Existence of a complex-balanced equilibrium]\label{prop.cbequi}
Let $\MM\in\R_+^m$ be an initial mass vector satisfying $\bm{\zeta}\MM=1$.
Then there exists a unique positive complex-balanced equilibrium $\bm{x}_\infty
\in\R_+^n$ satisfying \eqref{5.cbc} and
\begin{equation}\label{5.cb2}
  \Q\bm{x}_\infty = \MM\sum_{i=1}^n M_ix_{i\infty}, \quad \sum_{i=1}^n x_{i\infty}=1.
\end{equation}
\end{proposition}

The proof follows from the case of detailed balance with the help of the following
lemma.

\begin{lemma}\label{lem.om.rep}
Let $\bm{x}_\infty$ be a positive complex-balanced equilibrium. 
Then the following two statements are equivalent:

{\rm (i)} The vector $\bm{x}_*\in\R_+^n$ is a complex-balanced equilibrium.

{\rm (ii)} It holds for all $a=1,\ldots,N$:
$$
  \frac{\bm{x}_*^{\bm{y}_a}}{\bm{x}_*^{\bm{y}'_a}}
	= \frac{\bm{x}_\infty^{\bm{y}_a}}{\bm{x}_\infty^{\bm{y}'_a}}.
$$
\end{lemma}

\begin{proof}
Let (ii) hold. We compute
\begin{align*}
  \sum_{a:\bm{y}_a=\bm{y}}k^a\bm{x}_*^{\bm{y}_a}
	&= \sum_{a:\bm{y}_a=\bm{y}}k^a\bm{x}_\infty^{\bm{y}_a}
	\frac{\bm{x}_*^{\bm{y}_a}}{\bm{x}_\infty^{\bm{y}_a}}
	= \frac{\bm{x}_*^{\bm{y}}}{\bm{x}_\infty^{\bm{y}}}
	\sum_{a:\bm{y}_a=\bm{y}}k^a\bm{x}_\infty^{\bm{y}_a} \\
	&= \frac{\bm{x}_*^{\bm{y}}}{\bm{x}_\infty^{\bm{y}}}
	\sum_{b:\bm{y}'_b=\bm{y}}k^b\bm{x}_\infty^{\bm{y}_b}
	= \sum_{b:\bm{y}'_b=\bm{y}}k^b\bm{x}_\infty^{\bm{y}_b}
	\frac{\bm{x}_*^{\bm{y}'_b}}{\bm{x}_\infty^{\bm{y}'_b}}.
\end{align*}
Taking into account (ii), it follows that
$$
  \sum_{a:\bm{y}_a=\bm{y}}k^a\bm{x}_*^{\bm{y}_a}
	= \sum_{b:\bm{y}'_b=\bm{y}}k^b\bm{x}_\infty^{\bm{y}_b}
	\frac{\bm{x}_*^{\bm{y}_b}}{\bm{x}_\infty^{\bm{y}_b}}
	= \sum_{b:\bm{y}'_b=\bm{y}}k^b\bm{x}_*^{\bm{y}_b},
$$
i.e., $\bm{x}_*$ is a complex-balanced equilibrium.

To show that (i) implies (ii), let $\bm{x}_*$ be a complex-balanced
equilibrium. Then $r_i(\bm{x}_*)=0$ for all $i=1,\ldots,n$,
and the proof of Lemma \ref{lem.r2} shows that
$$
  0 = \sum_{i=1}^n\frac{r_i(\bm{x}_*)}{M_i}\ln\frac{x_{i*}}{x_{i\infty}}
	= -\sum_{a=1}^N k^a\bm{x}_\infty^{\bm{y}_a}\Psi\bigg(
	\frac{\bm{x}_*^{\bm{y}_a}}{\bm{x}_\infty^{\bm{y}_a}},
	\frac{\bm{x}_*^{\bm{y}'_a}}{\bm{x}_\infty^{\bm{y}'_a}}\bigg),
$$
where we recall that $\Psi(x,y)=x\ln(x/y)-x+y\ge 0$ and $\Psi(x,y)=0$ if and only
if $x=y$. The last property implies that
$\bm{x}_*^{\bm{y}_a}/\bm{x}_\infty^{\bm{y}_a} 
= \bm{x}_*^{\bm{y}'_a}/\bm{x}_\infty^{\bm{y}'_a}$, which is (ii).
\end{proof}

We prove a result similar to that one stated in Lemma \ref{lem.om.cl}.

\begin{lemma}\label{lem.om.cb}
The vector $\overline{\bm{\omega}}=(\overline{c_1},\ldots,\overline{c_n},\overline{c})
\in\R_+^{n+1}$ satisfies
\begin{equation}\label{5.sqrtom}
  \sqrt{\frac{\overline{\bm{\omega}}}{\bm{\omega}_\infty}}^{\bm{\mu}^a}
	= \sqrt{\frac{\overline{\bm{\omega}}}{\bm{\omega}_\infty}}^{\bm{\nu}^a}
	\quad\mbox{for all }a=1,\ldots,N, \quad
	\widehat{\Q}\overline{\bm{\omega}} = \wMM,
\end{equation}
if and only if $\overline{\bm{\omega}}=\bm{\omega}_\infty=(c_{1\infty},\ldots,
c_{n\infty},c_\infty)$ and 
$\bm{x}_\infty=(c_{1\infty}/c_\infty,\ldots,c_{n\infty}/c_\infty)$
is a complex-balanced equilibrium. Here, $c_\infty=\sum_{i=1}^n c_{i\infty}$ and 
$\widehat{\Q}$ and $\wMM$ are defined in \eqref{3.hat}.
\end{lemma}

\begin{proof}
Set $x_i=\overline{c_i}/\overline{c}$ for $i=1,\ldots,n$. Then the first
equation in \eqref{5.sqrtom} implies that, using definition \eqref{3.munu}
of $\bm{\mu}^a$ and $\bm{\nu}^a$,
$$
  \prod_{i=1}^n \frac{\overline{c_i}^{y_{i,a}}}{c_{i\infty}^{y_{i,a}}}
	= \prod_{i=1}^n\frac{\overline{c_i}^{y'_{i,a}}}{c_{i\infty}^{y'_{i,a}}}
	\frac{\overline{c}^{\gamma^a}}{c_\infty^{\gamma^a}},
	\quad\mbox{where }\gamma^a = \sum_{i=1}^n(y_{i,a}-y'_{i,a}).
$$
This is equivalent to 
$$
  \frac{\bm{x}^{\bm{y}_a}}{\bm{x}_\infty^{\bm{y}_a}}
	= \frac{\bm{x}^{\bm{y}'_a}}{\bm{x}_\infty^{\bm{y}'_a}}.
$$
We conclude from Lemma \ref{lem.om.rep} that $\bm{x}$ is a complex-balanced
equilibrium.
Furthermore, we have
$$
  \sum_{i=1}^n M_ix_i = \frac{1}{\overline{c}}\sum_{i=1}^n M_i\overline{c_i}
	= \frac{1}{\overline{c}}.
$$
Thus, we deduce from the conservation law 
$\widehat{\Q}\overline{\bm{\omega}} = \wMM$ that
$$
  \Q\bm{x} = \frac{1}{\overline{c}}\MM = \MM\sum_{i=1}^n M_ix_i.
$$
At this point, we can apply Proposition \ref{prop.cbequi} to infer the existence
of a unique vector $\bm{x}=\bm{x}_\infty$ which implies that
$\overline{\bm{\omega}}=\bm{\omega}_\infty$.
\end{proof}

Finally, we show an inequality which is related to that one in Lemma \ref{lem.Dreac.ge}.

\begin{lemma}
There exists a nonconstructive constant $C>0$ such that
$$
  \sum_{a=1}^N\bigg(\sqrt{\frac{\overline{\bm{\omega}}}{\bm{\omega}_\infty}}^{\bm{\mu}^a}
	- \sqrt{\frac{\overline{\bm{\omega}}}{\bm{\omega}_\infty}}^{\bm{\nu}^a}\bigg)^2 
	\ge C\sum_{i=1}^{n+1}\big(\overline{\omega_i}^{1/2}-\omega_{i\infty}^{1/2}\big)^2
$$
for all $\overline{\bm{\omega}}\in\R_+^{n+1}$ satisfying 
$\widehat{\Q}\overline{\bm{\omega}}=\wMM$.
\end{lemma}

\begin{proof}
We proceed similarly as in the proofs of Lemmas \ref{lem.delta} and \ref{lem.Dreac.ge}.
We need to show that
$$
  \lambda := \inf_{\overline{\bm{\omega}}\in\R_+^{n+1}:
	\widehat{\Q}\overline{\bm{\omega}}=\wMM}
	\frac{\sum_{a=1}^N\big(\sqrt{\overline{\bm{\omega}}/\bm{\omega}_\infty}^{\bm{\mu}^a}
	- \sqrt{\overline{\bm{\omega}}/\bm{\omega}_\infty}^{\bm{\nu}^a}\big)^2}{\sum_{i=1}^{n+1}
	\big(\overline{\omega_i}^{1/2}-\omega_{i\infty}^{1/2}\big)^2} > 0.
$$
In view of Lemma \ref{lem.om.cb} and the absence of boundary equilibria,
it holds $\lambda>0$ if and only if $\delta>0$, where
\begin{align*}
  \delta &= \liminf_{\widehat{\Q}\overline{\bm{\omega}}=\wMM,\,\overline{\bm{\omega}}
	\to\bm{\omega}_\infty}
	\frac{\sum_{a=1}^N\big(\sqrt{\overline{\bm{\omega}}/\bm{\omega}_\infty}^{\bm{\mu}^a}
	- \sqrt{\overline{\bm{\omega}}/\bm{\omega}_\infty}^{\bm{\nu}^a}\big)^2}{\sum_{i=1}^{n+1}
	\big(\overline{\omega_i}^{1/2}-\omega_{i\infty}^{1/2}\big)^2} \\
	&= \liminf_{\widehat{\Q}\overline{\bm{\omega}}=\wMM,\,\overline{\bm{\omega}}
	\to\bm{\omega}_\infty}
	\frac{2\sum_{a=1}^N\big(\sum_{i=1}^{n+1}(y_{i,a}-y'_{i,a})(\overline{\omega_i}
	- \omega_{i\infty})\omega_{i\infty}^{-1}\big)^2}{\sum_{i=1}^{n+1}(\overline{\omega_i}
	-\omega_{i\infty})^2\omega_{i\infty}^{-1}}.
\end{align*}
This follows from a Taylor expansion as in the proof of Lemma \ref{lem.delta}.
Now, we can follow exactly the arguments in the proof of Lemma \ref{lem.Dreac.ge}
to infer that $\delta>0$ and consequently $\lambda>0$, finishing the proof.
\end{proof}

The results in this subsection are sufficient to apply the proof of
Theorem \ref{thm.main}, thus leading to the following main theorem.

\begin{theorem}[Convergence to equilibrium for complex-balanced systems]\label{thm.main2}
\mbox{ } Let Assumptions \eqref{A1} and \eqref{A3} hold and let system 
\eqref{1.eq}, \eqref{5.reac} be complex balanced. Fix an initial mass vector
$\MM \in \R_+^m$ satisfying $\bm{\zeta}\MM = 1$. Then

{\rm (i)} There exists a global bounded weak solution 
$\bm{\rho}=(\rho_1,\ldots,\rho_n)^\top$ 
to \eqref{1.eq}, \eqref{1.bic} with reaction terms \eqref{5.reac}
in the sense of Theorem \ref{thm.ex}.

{\rm (ii)} There exists a unique positive complex-balanced equilibrium 
$\bm{x}_\infty\in\R_+^n$ satisfying \eqref{5.cbc} and \eqref{5.cb2}.

{\rm (iii)} Assume in addition that system \eqref{1.eq}, \eqref{5.reac} has no 
boundary equilibria. Then there exist constants $C>0$ and $\lambda>0$, which are 
constructive up to a finite-dimensional inequality, such that 
if $\bm{\rho}^0$ satisfies additionally $\mathbb Q\int_{\Omega}\bm{c}^0dz = \MM$, 
the following exponential convergence to equilibrium holds
$$
  \sum_{i=1}^n\|x_i(t)-x_{i\infty}\|_{L^p(\Omega)} \le Ce^{-\lambda t/(2p)}
  E[\bm{x}^0|\bm{x}_\infty]^{1/(2p)} ,\quad t>0,
$$
where $1\le p <\infty$,  $x_i=\rho_i/(cM_i)$ with $c=\sum_{i=1}^n\rho_i/M_i$, and
$E[\bm{x}|\bm{x}_\infty]$ is the relative entropy defined in \eqref{1.ent}, $\bm{\rho}$
is the solution constructed in {\rm (i)}, and $\bm{x}_\infty$ is 
constructed in {\rm (ii)}.
\end{theorem}


\begin{appendix}
\section{Proof of Lemma \ref{lem.tech}}\label{sec.tech}

The proof of Lemma \ref{lem.tech} is partially inspired by the proof of Lemma 2.7
in \cite{FeTa17a}. We divide the proof into two steps, which are presented
in Lemmas \ref{lem.ineq1} and \ref{lem.ineq2}. 
For convenience, we set $W_i:=\omega_i^{1/2}$ for $i=1,\ldots,n+1$ and use the notation
$$
  \bm{W} = (W_1,\ldots,W_{n+1}), \quad 
	\overline{\bm{W}} = (\overline{W_1},\ldots,\overline{W}_{n+1}).
$$
Moreover, we define
$$
  \delta_i(x) = W_i(x) - \overline{W_i} = W_i(x) - \int_\Omega W_idz, \quad
	x\in\Omega,\ i=1,\ldots,n+1.
$$

\begin{lemma}\label{lem.ineq1}
There exists a constant $C>0$ depending on $\Omega$, $n$, $N$, $k_f^a$, and
$k_b^a$ ($a=1,\ldots,N$) such that
\begin{equation}\label{a.ineq1}
	\widetilde{D}[\bm{\omega}]\ge C\sum_{a=1}^N\Big((k_f^a)^{1/2}
	\overline{\bm{W}}^{\bm{\mu}^a}
	- (k_b^a)^{1/2}\overline{\bm{W}}^{\bm{\nu}^a}\Big)^2
\end{equation}
where $\widetilde{D}$ is defined in \eqref{3.D}.
\end{lemma}

\begin{proof}
We use the elementary inequality $(x-y)\ln(x/y)\ge 4(\sqrt{x}-\sqrt{y})^2$ to obtain
$$
  \int_\Omega\big(k_f^a\bm{\omega}^{\bm{\mu}^a}-k_b^a\bm{\omega}^{\bm{\nu}^a}\big)
	\ln\frac{k_f^a\bm{\omega}^{\bm{\mu}^a}}{k_b^a\bm{\omega}^{\bm{\nu}^a}}dz
	\ge 4\int_\Omega\big(
	(k_f^a)^{1/2}\bm{W}^{\bm{\mu}^a}-(k_b^a)^{1/2}\bm{W}^{\bm{\nu}^a}\big)^2 dz.
$$
This gives
$$
 \widetilde{D}[\bm{\omega}] \ge \sum_{i=1}^{n+1}\|\na W_i\|_{L^2(\Omega)}^2
	+ 4\sum_{i=1}^{n+1}\big\|
	(k_f^a)^{1/2}\bm{W}^{\bm{\mu}^a}-(k_b^a)^{1/2}\bm{W}^{\bm{\nu}^a}\big\|_{L^2(\Omega}^2.
$$
The Poincar\'e inequality 
$$
  \|\na W_i\|_{L^2(\Omega)}^2 \ge C_P\|\delta_i\|_{L^2(\Omega)}^2
$$
then shows that
\begin{equation}\label{a.aux1}
  \widetilde{D}[\bm{\omega}] \ge C_P\sum_{i=1}^{n+1}\|\delta_i\|_{L^2(\Omega)}^2
	+ 4\sum_{i=1}^{n+1}\big\|
	(k_f^a)^{1/2}\bm{W}^{\bm{\mu}^a}-(k_b^a)^{1/2}\bm{W}^{\bm{\nu}^a}\big\|_{L^2(\Omega)}^2.
\end{equation}
Let $L>0$. We split $\Omega$ into the two domains
$$
  \Omega_L = \big\{x\in\Omega:|\delta_i(x)|\le L\mbox{ for }i=1,\ldots,n+1\big\},
	\quad \Omega_L^c = \Omega\backslash\Omega_L.
$$
By Taylor expansion, we may write $W_i^{\mu_i^a}=(\overline{W_i}+\delta_i)^{\mu_i^a}
= \overline{W_i}^{\mu_i^a} + R_i^*(\overline{W_i},\delta_i)\delta_i$, where
$R_i^*$ depends continuously on $\overline{W_i}$ and $\delta_i$. Therefore,
\begin{align*}
  \big\|(k_f^a)^{1/2}&\bm{W}^{\bm{\mu}^a}
	- (k_b^a)^{1/2}\bm{W}^{\bm{\nu}^a}\big\|_{L^2(\Omega)}^2 \\
	&\ge \int_{\Omega_L}\bigg|(k_f^a)^{1/2}\prod_{i=1}^{n+1}
	(\overline{W_i}+\delta_i)^{\mu_i^a}
	- (k_b^a)^{1/2}\prod_{i=1}^{n+1}(\overline{W_i}+\delta_i)^{\nu_i^a}\bigg|^2 dz \\
	&= \int_{\Omega_L}\bigg|(k_f^a)^{1/2}\prod_{i=1}^{n+1}\big(\overline{W_i}^{\mu_i^a} 
	+ R_i^*\delta_i\big) - (k_b^a)^{1/2}\prod_{i=1}^{n+1}\big(\overline{W_i}^{\nu_i^a} 
	+ R_i^*\delta_i\big)\bigg|^2 dz \\
  &= \int_{\Omega_L}\bigg|(k_f^a)^{1/2}\overline{\bm{W}}^{\bm{\mu}^a}
	- (k_b^a)^{1/2}\overline{\bm{W}}^{\bm{\nu}^a} + Q^*\sum_{i=1}^{n+1}\delta_i
	\bigg|^2 dz,
\end{align*}
where $Q^*$ depends continously on $R_1^*,\ldots,R_{n+1}^*$ and 
$\delta_1,\ldots,\delta_{n+1}$.
With the inequalities $(x+y)^2\ge\frac12(x^2-y^2)$ and $(\sum_{i=1}^{n+1}x_i)^2
\le (n+1)\sum_{i=1}^{n+1}x_i^2$, we estimate
\begin{align*}
  \big\|(k_f^a)^{1/2}&\bm{W}^{\bm{\mu}^a}
	- (k_b^a)^{1/2}\bm{W}^{\bm{\nu}^a}\big\|_{L^2(\Omega)}^2 \\
	&\ge \frac12|\Omega_L|\Big((k_f^a)^{1/2}\overline{\bm{W}}^{\bm{\mu}^a}
	- (k_b^a)^{1/2}\overline{\bm{W}}^{\bm{\nu}^a}\Big)^2
	- \int_{\Omega_L}(Q^*)^2(n+1)\sum_{i=1}^{n+1}|\delta_i|^2 dz \\
	&\ge \frac12|\Omega_L|\Big((k_f^a)^{1/2}\overline{\bm{W}}^{\bm{\mu}^a}
	- (k_b^a)^{1/2}\overline{\bm{W}}^{\bm{\nu}^a}\Big)^2
	- C(L)(n+1)\sum_{i=1}^{n+1}\|\delta_i\|_{L^2(\Omega)}^2,
\end{align*}
where we used the bounds $|\delta_i|\le L$ in $\Omega_L$ and $\overline{W_i}\le C$ 
in $\Omega$ to estimate $Q^*$. Summing over $a=1,\ldots,N$, this gives
\begin{align}
  \sum_{a=1}^N\big\|(k_f^a)^{1/2}\bm{W}^{\bm{\mu}^a}
	- (k_b^a)^{1/2}\bm{W}^{\bm{\nu}^a}\big\|_{L^2(\Omega)}^2 
  &\ge \frac12|\Omega_L|\sum_{a=1}^N\Big((k_f^a)^{1/2}\overline{\bm{W}}^{\bm{\mu}^a}
	- (k_b^a)^{1/2}\overline{\bm{W}}^{\bm{\nu}^a}\Big)^2 \nonumber \\
	&\phantom{xx}{}- C(L)N(n+1)\sum_{i=1}^{n+1}\|\delta_i\|_{L^2(\Omega)}^2.
	\label{a.aux2}
\end{align}
In $\Omega_L^c$, we wish to estimate $\|\delta_i\|_{L^2(\Omega)}$ from below. 
For this, we observe that
$$
  \sum_{a=1}^N\Big((k_f^a)^{1/2}\overline{\bm{W}}^{\bm{\mu}^a}
	- (k_b^a)^{1/2}\overline{\bm{W}}^{\bm{\nu}^a}\Big)^2
	\le C.
$$
Then, since $\sum_{i=1}^{n+1}|\delta_i|\ge L$ on $\Omega_L^c$,
\begin{align}
  \sum_{i=1}^{n+1}\|\delta_i\|_{L^2(\Omega)}^2 
	&\ge \sum_{i=1}^{n+1}\int_{\Omega_L^c}|\delta_i|^2 dz 
	\ge \frac{1}{n+1}\int_{\Omega_L^c}\bigg(\sum_{i=1}^{n+1}|\delta_i|\bigg)^2 dz 
	\nonumber \\
	&\ge \frac{L^2|\Omega_L^c|}{n+1}
	\ge \frac{L^2|\Omega_L^c|}{(n+1)C}
	\sum_{a=1}^N\Big((k_f^a)^{1/2}\overline{\bm{W}}^{\bm{\mu}^a}
	- (k_b^a)^{1/2}\overline{\bm{W}}^{\bm{\nu}^a}\Big)^2. \label{a.aux3}
\end{align}
Inserting \eqref{a.aux2} and \eqref{a.aux3} into \eqref{a.aux1}, it follows for
any $\theta\in(0,1)$ that
\begin{align*}
  \widetilde{D}[\bm{\omega}] &\ge C_P\sum_{i=1}^{n+1}\|\delta_i\|_{L^2(\Omega)}^2
	+ 4\theta\sum_{i=1}^{n+1}\big\|
	(k_f^a)^{1/2}\bm{W}^{\bm{\mu}^a}-(k_b^a)^{1/2}\bm{W}^{\bm{\nu}^a}
	\big\|_{L^2(\Omega)}^2 \\
	&\ge \frac{C_P}{2}\sum_{i=1}^{n+1}\|\delta_i\|_{L^2(\Omega)}^2
	+ \frac{C_P}{2}\frac{L^2|\Omega_L^c|}{(n+1)C}
	\sum_{a=1}^N\Big((k_f^a)^{1/2}\overline{\bm{W}}^{\bm{\mu}^a}
	- (k_b^a)^{1/2}\overline{\bm{W}}^{\bm{\nu}^a}\Big)^2 \\
	&\phantom{xx}{}
	+ 2\theta|\Omega_L|\sum_{a=1}^N\Big((k_f^a)^{1/2}\overline{\bm{W}}^{\bm{\mu}^a}
	- (k_b^a)^{1/2}\overline{\bm{W}}^{\bm{\nu}^a}\Big)^2
  - 4\theta C(L)(n+1)\sum_{i=1}^{n+1}\|\delta_i\|_{L^2(\Omega)}^2 \\
	&\ge C\sum_{a=1}^N\Big((k_f^a)^{1/2}\overline{\bm{W}}^{\bm{\mu}^a}
	- (k_b^a)^{1/2}\overline{\bm{W}}^{\bm{\nu}^a}\Big)^2,
\end{align*}
where we have chosen $\theta>0$ sufficiently small in the last step.
This finishes the proof.
\end{proof}

\begin{lemma}\label{lem.ineq2}
There exists a constant $C>0$ depending on $\Omega$, $n$, $N$, $k_f^a$, and
$k_b^a$ ($a=1,\ldots,N$) such that
\begin{equation}\label{a.ineq2}
\begin{aligned}
  \sum_{i=1}^{n+1}|\na\omega_i^{1/2}|^2 dz
	&+ \sum_{a=1}^N\Big((k_f^a)^{1/2}\overline{\bm{W}}^{\bm{\mu}^a}
	- (k_b^a)^{1/2}\overline{\bm{W}}^{\bm{\nu}^a}\Big)^2 \\
	&\ge C\sum_{a=1}^N\Big((k_f^a)^{1/2}\sqrt{\overline{\bm{\omega}}}^{\bm{\mu}^a}
	- (k_b^a)^{1/2}\sqrt{\overline{\bm{\omega}}}^{\bm{\nu}^a}\Big)^2.
\end{aligned}
\end{equation}
\end{lemma}

\begin{proof}
It follows from
$$
  \|\delta_i\|_{L^2(\Omega)}^2 = \|W_i-\overline{W_i}\|_{L^2(\Omega)}^2
	= \overline{\omega_i} - \overline{W_i}^2
	= \big(\sqrt{\overline{\omega_i}}-\overline{W_i}\big)
	\big(\sqrt{\overline{\omega_i}}+\overline{W_i}\big)
$$
that
$$
  \overline{W_i} = \sqrt{\overline{\omega_i}} - Z_i\|\delta_i\|_{L^2(\Omega)},
	\quad\mbox{where }Z_i=\frac{\|\delta_i\|_{L^2(\Omega)}}{\sqrt{\overline{\omega_i}}
	+\overline{W_i}}\ge 0.
$$
Since
$$
  Z_i^2 = \frac{\|\delta_i\|_{L^2(\Omega)}^2}{(\sqrt{\overline{\omega_i}}
	+\overline{W_i})^2}
	= \frac{\overline{\omega_i}-\overline{W_i}^2}{(\sqrt{\overline{\omega_i}}
	+\overline{W_i})^2}
	= \frac{\sqrt{\overline{\omega_i}}-\overline{W_i}}{\sqrt{\overline{\omega_i}}
	+\overline{W_i}} \le 1,
$$
we infer that $0\le Z_i\le 1$. 

We continue by performing a Taylor expansion:
$$
  \overline{\bm{W}}^{\bm{\mu}^a}
	= \prod_{i=1}^{n+1}\big(\sqrt{\overline{\omega_i}} - Z_i\|\delta_i\|_{L^2(\Omega)}
	\big)^{\mu_i^a}
	= \prod_{i=1}^{n+1}\big(\sqrt{\overline{\omega_i}}^{\mu_i^a}
	+ R_i^*\|\delta_i\|_{L^2(\Omega)}\big),
$$
where $R_i^*$ depends continuously on $Z_i$ and $\|\delta_i\|_{L^2(\Omega)}$.
Therefore, with another function $S^*$ depending continuously on $Z_i$ 
and $\|\delta_i\|_{L^2(\Omega)}$,
$$
  \overline{\bm{W}}^{\bm{\mu}^a} = \sqrt{\overline{\bm{\omega}}}^{\bm{\mu}^a}
	+ S^*\sum_{i=1}^{n+1}\|\delta_i\|_{L^2(\Omega)}.
$$
This shows that
\begin{align*}
  \sum_{a=1}^N & \Big((k_f^a)^{1/2}\overline{\bm{W}}^{\bm{\mu}^a}
	- (k_b^a)^{1/2}\overline{\bm{W}}^{\bm{\nu}^a}\Big)^2 \\
  &= \sum_{a=1}^N\bigg((k_f^a)^{1/2}\sqrt{\overline{\bm{\omega}}}^{\bm{\mu}^a}
	- (k_b^a)^{1/2}\sqrt{\overline{\bm{\omega}}}^{\bm{\nu}^a}
	+ \big((k_f^a)^{1/2}-(k_b^a)^{1/2}\big)
	S^*\sum_{i=1}^{n+1}\|\delta_i\|_{L^2(\Omega)}\bigg)^2 \\
	&\ge \frac12\sum_{a=1}^N\Big((k_f^a)^{1/2}\sqrt{\overline{\bm{\omega}}}^{\bm{\mu}^a}
	- (k_b^a)^{1/2}\sqrt{\overline{\bm{\omega}}}^{\bm{\nu}^a}\Big)^2
	- C(n,N)(S^*)^2\sum_{i=1}^{n+1}\|\delta_i\|_{L^2(\Omega)}^2.
\end{align*}
Then, by the Poincar\'e inequality with constant $C_P$, for some $\theta\in(0,1)$,
\begin{align*}
  \sum_{i=1}^{n+1}|\na\omega_i^{1/2}|^2 dz
	&+ \sum_{a=1}^N\Big((k_f^a)^{1/2}\overline{\bm{W}}^{\bm{\mu}^a}
	- (k_b^a)^{1/2}\overline{\bm{W}}^{\bm{\nu}^a}\Big)^2 \\
	&\ge C_P\sum_{i=1}^{n+1}\|\delta_i\|_{L^2(\Omega)}^2
	+ \theta\sum_{a=1}^N\Big((k_f^a)^{1/2}\sqrt{\overline{\bm{\omega}}}^{\bm{\mu}^a}
	- (k_b^a)^{1/2}\sqrt{\overline{\bm{\omega}}}^{\bm{\nu}^a}\Big)^2 \\
	&\ge C_P\sum_{i=1}^{n+1}\|\delta_i\|_{L^2(\Omega)}^2
	+ \frac{\theta}{2}
	\sum_{a=1}^N\Big((k_f^a)^{1/2}\sqrt{\overline{\bm{\omega}}}^{\bm{\mu}^a}
	- (k_b^a)^{1/2}\sqrt{\overline{\bm{\omega}}}^{\bm{\nu}^a}\Big)^2 \\
	&\phantom{xx}{}- \theta C(n,N)(S^*)^2\sum_{i=1}^{n+1}\|\delta_i\|_{L^2(\Omega)}^2 \\
	&\ge \frac{\theta}{2}
	\sum_{a=1}^N\Big((k_f^a)^{1/2}\sqrt{\overline{\bm{\omega}}}^{\bm{\mu}^a}
	- (k_b^a)^{1/2}\sqrt{\overline{\bm{\omega}}}^{\bm{\nu}^a}\Big)^2.
\end{align*}
The last step follows after choosing $\theta>0$ sufficiently small. This is
possible since $S^*$ is bounded. The proof is complete.
\end{proof}

\begin{proof}[Proof of Lemma \ref{lem.tech}]
Applying first
\eqref{a.ineq1} and then \eqref{a.ineq2} leads to
\begin{align*}
  \widetilde{D}[\bm{\omega}]
	&\ge \frac{C}{2}\sum_{i=1}^{n+1}\int_\Omega |\na\omega_i^{1/2}|^2 dz \\
	&\phantom{xx}{}+ \frac{C}{2}\bigg(\sum_{i=1}^{n+1}\int_\Omega |\na\omega_i^{1/2}|^2 dz
	+ \sum_{a=1}^N\int_\Omega\big(k_f^a\bm{\omega}^{\bm{\mu}^a}
	-k_b^a\bm{\omega}^{\bm{\nu}^a}\big)
	\ln\frac{k_f^a\bm{\omega}^{\bm{\mu}^a}}{k_b^a\bm{\omega}^{\bm{\nu}^a}}dz\bigg) \\
	&\ge \frac{C}{2}\sum_{i=1}^{n+1}\int_\Omega |\na\omega_i^{1/2}|^2 dz
	+ C\sum_{a=1}^N\Big((k_f^a)^{1/2}\overline{\bm{W}}^{\bm{\mu}^a}
	- (k_b^a)^{1/2}\overline{\bm{W}}^{\bm{\nu}^a}\Big)^2 \\
	&\ge  C\sum_{a=1}^N\Big((k_f^a)^{1/2}\sqrt{\overline{\bm{\omega}}}^{\bm{\mu}^a}
	- (k_b^a)^{1/2}\sqrt{\overline{\bm{\omega}}}^{\bm{\nu}^a}\Big)^2.
\end{align*}
The proof is finished.
\end{proof}

\end{appendix}

\section*{Compliance with ethical standards}

{\bf Conflict of interest.} The authors declare that they have no conflict of interest.


\end{document}